\documentclass{amsart}
\usepackage{amssymb,amsmath,amsthm,hyperref}
\newtheorem{theorem}{Theorem}[section]
\newtheorem{lemma}[theorem]{Lemma}
\newtheorem{corollary}[theorem]{Corollary}

\theoremstyle{definition}

\theoremstyle{remark}

\numberwithin{equation}{section}

\newcommand{\spt}{\mbox{\rm spt}}

\newcommand{\SL}{\mbox{SL}}

\DeclareSymbolFont{AMSb}{U}{msb}{m}{n}
\DeclareMathSymbol{\Z}{\mathalpha}{AMSb}{"5A}

\DeclareMathSymbol{\nmid}{\mathrel}{AMSb}{"2D}
\DeclareSymbolFont{AMSb}{U}{msb}{m}{n}
\DeclareMathSymbol{\C}{\mathalpha}{AMSb}{"43}
\DeclareMathSymbol{\F}{\mathalpha}{AMSb}{"46}
\DeclareMathSymbol{\N}{\mathalpha}{AMSb}{"4E}
\DeclareMathSymbol{\Q}{\mathalpha}{AMSb}{"51}
\DeclareMathSymbol{\R}{\mathalpha}{AMSb}{"52}
\DeclareMathSymbol{\Z}{\mathalpha}{AMSb}{"5A}

\allowdisplaybreaks

\begin{document}
\newcommand{\beqs}{\begin{equation*}}
\newcommand{\eeqs}{\end{equation*}}
\newcommand{\beq}{\begin{equation}}
\newcommand{\eeq}{\end{equation}}
\newcommand\mylabel[1]{\label{#1}}
\newcommand\eqn[1]{(\ref{eq:#1})}
\newcommand\thm[1]{\ref{thm:#1}}
\newcommand\lem[1]{\ref{lem:#1}}
\newcommand\propo[1]{\ref{propo:#1}}
\newcommand\cor[1]{\ref{cor:#1}}
\newcommand\sect[1]{\ref{sec:#1}}
\newcommand\leg[2]{\genfrac{(}{)}{}{}{#1}{#2}} 
\newcommand\ord{\mbox{ord}\,}
\newcommand\mytimes{\cdot}
\newcommand{\stroke}{\,\mid\,}
\newcommand{\Tr}{\mbox{Tr}}

\title[Andrews' spt-function modulo powers of $5$, $7$ and $13$]{
Congruences for Andrews' spt-function \\
modulo powers of $5$, $7$ and $13$}


\author{F.~G.~Garvan}
\address{Department of Mathematics, University of Florida, Gainesville,
Florida 32611-8105}
\email{fgarvan@ufl.edu}          
\thanks{The author was supported in part by NSA Grant H98230-09-1-0051.
The first draft of this paper was written September 20, 2010.
}



\date{\today}  



\begin{abstract}
Congruences are found modulo powers of $5$, $7$ and $13$ for
Andrews' smallest parts partition function $\spt(n)$. These congruences
are reminiscent of Ramanujan's partition congruences modulo
powers of $5$, $7$ and $11$. Recently, Ono proved explicit Ramanujan-type
congruences for $\spt(n)$ modulo $\ell$ for all primes $\ell\ge5$ which
were conjectured earlier by the author. We extend Ono's method to
handle the powers of  $5$, $7$ and $13$ congruences. We need the
theory of weak Maass forms as well as certain classical modular
equations for the Dedekind eta-function.
\end{abstract}

\maketitle

\section{Introduction} \mylabel{sec:intro}

Andrews \cite{An08b} defined the function $\spt(n)$ as the number 
of 
smallest parts in the partitions of $n$. He related this
function to the second rank moment. He also proved some surprising congruences
mod $5$, $7$ and $13$. Namely, he showed that
\begin{equation}
\spt(n) = n p(n) - \frac{1}{2} N_2(n),
\mylabel{eq:sptid}
\end{equation}
where $N_2(n)$ is the second rank moment function \cite{At-Ga}
and $p(n)$ is the number of 
 partitions of $n$, and he proved that
\begin{align}
\spt(5n+4) &\equiv 0 \pmod{5},
\mylabel{eq:spt5cong}\\
\spt(7n+5) &\equiv 0 \pmod{7},
\mylabel{eq:spt7cong}\\
\spt(13n+6) &\equiv 0 \pmod{13}.   
\mylabel{eq:spt13cong}
\end{align}
Bringmann \cite{Br08} studied analytic, arithmetic and asymptotic
properties of the generating function for the second rank moment as
a quasi-weak Maass form. Further congruence properties of Andrews'
spt-function were found by the author \cite{Ga10a}, Folsom and Ono \cite{Fo-On}
and Ono \cite{On10}. In particular, Ono \cite{On10}
proved that if $\leg{1-24n}{\ell}=1$ then
\beq
\spt(\ell^2 n - \tfrac{1}{24}(\ell^2-1)) \equiv 0 \pmod{\ell},
\mylabel{eq:sptellcong}
\eeq
for any prime $\ell \ge 5$. This amazing result was originally conjectured
by the author\footnote{The congruence \eqn{sptellcong} was first
conjectured by the author in a Colloquium given at the University of
Newcastle, Australia on July 17, 2008.}.
Earlier special cases were observed by Tina Garrett \cite{Ga-PC2007}
and her students.

We prove some suprising congruences for $\spt(n)$ modulo powers
of $5$, $7$ and $13$. For $a$, $b$, $c\ge3$,
\begin{align}
\spt(5^a n + \delta_a) + 5\, \spt(5^{a-2} n + \delta_{a-2}) &\equiv 0  \pmod{5^{2a-3}},
\mylabel{eq:spt5acong}\\
\spt(7^b n + \lambda_b) + 7\, \spt(7^{b-2} n + \lambda_{b-2}) &\equiv 0  \pmod{7^{\lfloor\frac{1}{2}(3b-2)\rfloor}},
\mylabel{eq:spt7bcong}\\
\spt(13^c n + \gamma_c) - 13\, \spt(13^{c-2} n + \gamma_{c-2}) &\equiv 0  \pmod{13^{c-1}},   
\mylabel{eq:spt13ccong}
\end{align}
where $\delta_a$, $\lambda_b$ and $\gamma_c$ are the 
least nonnegative residues of the reciprocals of $24$ mod
$5^a$, $7^b$ and $13^c$ respectively.  This together with \eqn{spt5cong}--\eqn{spt13cong}
implies that
\begin{align}
\spt(5^{a}n + \delta_{a}) &\equiv 0
\pmod{5^{\lfloor \frac{a+1}{2}\rfloor}},
\mylabel{eq:introspt5alittlecong}\\
\spt(7^{b}n + \lambda_{b}) &\equiv 0
\pmod{7^{\lfloor \frac{b+1}{2}\rfloor}},
\mylabel{eq:introspt7alittlecong}\\
\spt(13^{c}n + \gamma_{c}) &\equiv 0
\pmod{13^{\lfloor \frac{c+1}{2}\rfloor}},
\mylabel{eq:introspt13alittlecong}
\end{align}
for  $a$, $b$, $c\ge1$. These congruences are reminiscent
of Ramanujan's partition congruences for powers of $5$, $7$ and $11$:
\begin{align}
p(5^{a}n + \delta_{a}) &\equiv 0
\pmod{5^{a}},                                      
\mylabel{eq:ptn5cong}\\
p(7^{b}n + \lambda_{b}) &\equiv 0
\pmod{7^{\lfloor \frac{b+2}{2}\rfloor}},
\mylabel{eq:ptn7cong}\\
p(11^{c}n + \varphi_{c}) &\equiv 0
\pmod{11^{c}},
\mylabel{eq:ptn11cong}
\end{align}
for all  $a$, $b$, $c\ge1$. Here $\varphi_c$ is the reciprocal of $24$
mod $11^c$. The congruences mod powers of $5$ and $7$ were proved
by Watson \cite{Wa38}, although many of the details had been
worked out earlier by Ramanujan in an unpublished manuscript.
The powers of $11$ congruence was proved by Atkin \cite{At67}.

Following Ono \cite{On10}, we define
\beq
 \mathbf{a}(n) := 12 \, \spt(n)  + (24n -1) p(n),
\mylabel{eq:adef}
\eeq
for $n\ge 0$, and define
\beq
\alpha(z) := \sum_{n\ge0}  \mathbf{a}(n) q^{n - \tfrac{1}{24}},
\mylabel{eq:alphadef}
\eeq
where as usual $q = \exp(2\pi i z)$ and $\Im(z) > 0$.
We note that $\spt(0)=0$ and $p(0)=1$.
Bringmann \cite{Br08} showed that $\alpha(24z)$ is the
holomorphic part of a weight $\tfrac{3}{2}$ weak Maass form.
Using this observation and the idea of using 
the weight $\frac{3}{2}$ Hecke operator 
$T(\ell^2)$
to annihilate the nonholomorphic part 
enabled Ono \cite{On10} to prove the general congruence \eqn{sptellcong}.
We use a similar idea. 
Instead of a Hecke operator we use Atkin's $U(\ell)$ operator
to annihilate the nonholomorphic part.  

We show that
\begin{align}
\mathbf{a}(5^a n + \delta_a) + 5\, \mathbf{a}(5^{a-2} n + \delta_{a-2}) &\equiv 0  \pmod{5^{\lfloor\frac{1}{2}(5a-7)\rfloor}},
\mylabel{eq:introa5acong}\\
\mathbf{a}(7^b n + \lambda_b) + 7\, \mathbf{a}(7^{b-2} n + \lambda_{b-2}) &\equiv 0  \pmod{7^{\lfloor\frac{1}{2}(3b-2)\rfloor}},
\mylabel{eq:introa7bcong}\\
\mathbf{a}(13^c n + \gamma_c) - 13\, \mathbf{a}(13^{c-2} n + \gamma_{c-2}) &\equiv 0  \pmod{13^{c-1}},   
\mylabel{eq:introa13ccong}
\end{align}
for all $a$, $b$, $c\ge 3$. We note that \eqn{introa5acong} is a stronger
congruence than \eqn{spt5acong}. The congruences \eqn{spt5acong}--\eqn{spt7bcong}
follow from \eqn{introa5acong}--\eqn{introa7bcong} and Ramanujan's partition
congruences for powers of $5$ and $7$ that were first proved by 
Watson \cite{Wa38}.
The congruence \eqn{spt13ccong} follows easily from \eqn{introa13ccong}.

Let $\ell \ge 5$ be prime.
In Section \sect{atkin} we use results of Bringmann  \cite{Br08}
to show how Atkin's $U(\ell)$ operator
can be used to annihilate the nonholomorphic part of
the weight $\tfrac{3}{2}$ weak Maass form that corresponds to the
function $\alpha(24z)$, and prove that the function
\beq
\alpha_\ell(z) := 
\sum_{n=0}^\infty
\left(  \mathbf{a}(\ell n - \tfrac{1}{24}(\ell^2-1)) - \chi_{12}(\ell) \, \ell \, \mathbf{a}\left ( \frac{n}{\ell}\right)\right)
q^{n- \tfrac{\ell}{24}}
\mylabel{eq:alphaelldef}
\eeq
is a weakly holomorphic  weight $\tfrac{3}{2}$ modular form on $\Gamma_0(\ell)$.
Here $\chi_{12}$ is the character given below in \eqn{chi12}, and
we note $ \mathbf{a}(n)=0$ if n is not a nonnegative integer.
We determine the multiplier of this form and exact information about
the orders at cusps. See Theorem \thm{Gell}.
This enables us to prove identities such as
\beq
\alpha_5(z) = \sum_{n=0}^\infty
\left(  \mathbf{a}(5 n - 1)  + 5\, \mathbf{a}\left ( \frac{n}{5}\right)\right)
q^{n- \tfrac{5}{24}} 
=
\frac{5}{4} \frac{ (5{E}_{2}(5z) - E_2(z))}{\eta(5z)}\left( 
125\frac{\eta(5z)^{6}}{\eta(z)^6} - 1\right),
\mylabel{eq:introa5n1id}
\eeq
where $E_2(z)$ is the usual quasimodular Eisenstein series of weight $2$,
and $\eta(z)$ is the Dedekind eta-function. We then use
Watson's \cite{Wa38} and Atkin's \cite{At68} method of
modular equations to prove the congruences 
\eqn{introa5acong}--\eqn{introa13ccong}. These details are carried out
in Section \sect{congs}.
In Section \sect{sptell} we improve some results in
\cite{Ga10a} and \cite{Br-Ga-Ma}
on $\spt(\ell n - \tfrac{1}{24}(\ell^2-1))$ and 
$N_2(\ell n - \tfrac{1}{24}(\ell^2-1))$ modulo $\ell$.


\section{The Atkin operator $U_{\ell}^*$} \mylabel{sec:atkin}

In this section we prove that the function $\alpha_\ell(z)$, which is
defined in \eqn{alphaelldef} is 
a weakly holomorphic  weight $\tfrac{3}{2}$ modular form on $\Gamma_0(\ell)$ 
when $\ell\ge5$ is prime. The proof uses results of Bringmann \cite{Br08}
and the idea of using the Atkin operator $U_\ell$ to annihilate
the nonholomorphic part of a certain weak Maass form.

Following Bringmann \cite{Br08} and Ono \cite{On10} we define
\beq
\mathcal{M}(z) := \alpha(24z) - \frac{3i}{\pi\sqrt{2}} \, 
\int_{-\overline{z}}^{i\infty} \frac{\eta(24\tau) \, d\tau}{(-i(\tau+z))^{\tfrac32}},
\mylabel{eq:Mdef}
\eeq
where $\eta(z) := q^{\tfrac{1}{24}}\prod_{n=1}^\infty(1-q^n)$ is the Dedekind
eta-function and $\alpha(z)$ is defined in \eqn{alphadef}. Then $\mathcal{M}(z)$
is a weight $\tfrac{3}{2}$ harmonic Maass form on $\Gamma_0(576)$ with Nebentypus
$\chi_{12}$ where
\beq
\chi_{12}(n) =
\begin{cases}
1 & \mbox{if $n\equiv\pm1\pmod{12}$,}\\
-1 & \mbox{if $n\equiv\pm5\pmod{12}$,}\\
0 & \mbox{otherwise.}
\end{cases}
\mylabel{eq:chi12}
\eeq
Let
\beq
\mathcal{N}(z) = - \frac{3i}{\pi\sqrt{2}} 
\int_{-\overline{z}}^{i\infty} \frac{\eta(24\tau) \, d\tau}{(-i(\tau+z))^{\tfrac32}}
= \frac{3}{\pi\sqrt{2}}
\int_y^\infty \frac{\eta(24(-x+it))\,dt}{(y+t)^{3/2}},
\mylabel{eq:Nint}
\eeq              
where $z=x+iy$, $y>0$, so that
\beq
\mathcal{M}(z) = \alpha(24z) + \mathcal{N}(z).
\mylabel{eq:Mdef2}
\eeq
We define
\beq
\mathcal{A}(z) := \mathcal{M}\left(\frac{z}{24}\right).
\mylabel{eq:Azdef}
\eeq

The following theorem  follows in a straightforward way from the work 
of Bringmann \cite{Br08}.

\begin{theorem}
\mylabel{thm:Atrans}
$$
\mathcal{A}\left(\frac{az+b}{cz+d}\right) = \frac{ (cz+d)^{3/2} }{\nu_{\eta}(A)}
\mathcal{A}(z),
$$
where $A=
\begin{pmatrix}
a & b \\
c & d
\end{pmatrix} \in \SL_2(\Z)$, and
$\nu_{\eta}(A)$ is the eta-multiplier.
\end{theorem}

\noindent
{\it Remark}.        
When defining $z^{3/2}$ we use the principal branch; i.e.
for $z = r e^{i\theta}$, $r>0$, $-\pi \le \theta < \pi$, we take
$z^{3/2} = r^{3/2} e^{3i\theta/2}$.
\begin{proof}
We note that
\beq
\sum_{n=0}^\infty (24n -1) p(n) q^{n - \tfrac{1}{24}} = - \frac{E_2(z)}{\eta(z)},
\mylabel{eq:dpid}
\eeq
where $E_2(z) = 1 - 24 \sum_{n=1}^\infty \sigma(n) q^n$ is a quasi-modular
form that satisfies
\beq
E_2\left(\frac{az+b}{cz+d}\right) = (cz+d)^2 E_2(z)  - \frac{6iz}{\pi} c(cz+d).
\mylabel{eq:E2trans}
\eeq
Using \eqn{E2trans} and Corollary 4.3 and Lemma 4.4 in \cite{Br08},   
$$
\mathcal{M}\left(-\frac{1}{z}\right) = \frac{-(-iz)^{3/2}}{48\sqrt{6}} 
\mathcal{M}\left(\frac{z}{576}\right),
$$
and hence
$$
\mathcal{A}\left(-\frac{1}{z}\right) = -(-iz)^{3/2}  \mathcal{A}(z)
= e^{\pi i/4} z^{3/2}  \mathcal{A}(z).
$$
Therefore,
$$
\mathcal{A}(Sz) = \frac{z^{3/2}}{\nu_\eta(S)} \mathcal{A}(z),
$$
where $S = \begin{pmatrix} 0 & -1 \\ 1 & 0 \end{pmatrix}$.
From \eqn{alphadef}, \eqn{Nint} and \eqn{Mdef2}
\begin{align*}
\mathcal{M}(z + \tfrac{1}{24}) &= e^{-\pi i/12} \mathcal{M}(z),\\
\mathcal{N}(z + \tfrac{1}{24}) &= e^{-\pi i/12} \mathcal{N}(z),\\
\mathcal{A}(z + 1) &= e^{-\pi i/12} \mathcal{A}(z),\\
\mathcal{A}(Tz) &= \frac{1}{\nu_\eta(T)} \mathcal{A}(z),
\end{align*}
where $T = \begin{pmatrix} 1 & 1 \\ 0 & 1 \end{pmatrix}$.
Since $S$, $T$ generate $\SL_2(\Z)$ the result follows.
\end{proof}

In what follows $\ell \ge 5$ is prime. We let $d_\ell$ denote the
least nonnegative residue of the reciprocal of $24$ mod $\ell$ so
that $24d_\ell\equiv1\pmod{\ell}$. We define
\beq
r_\ell := \frac{24 d_\ell -1}{\ell},\qquad 
r_\ell^* := \frac{24 d_\ell + \ell^2 -1}{24\ell}, \qquad
s_\ell := \frac{(\ell^2-1)}{24}.
\mylabel{eq:rsdefs}
\eeq
so that
\beq
\alpha_\ell(z) := 
\sum_{n=-r_\ell^*} 
\left(  \mathbf{a}(\ell n + d_\ell) - \chi_{12}(\ell) \, \ell \, \mathbf{a}\left ( \frac{n+r_\ell^*}{\ell}\right)\right)
q^{n+ \tfrac{r_\ell}{24}} 
=
\sum_{n=0}^\infty
\left(  \mathbf{a}(\ell n - s_\ell) - \chi_{12}(\ell) \, \ell \, \mathbf{a}\left ( \frac{n}{\ell}\right)\right)
q^{n- \tfrac{\ell}{24}}.
\mylabel{eq:alphaelldefB}
\eeq
For a function $G(z)$ we define the Atkin-type operator 
$U_\ell^*$ by
\beq
U_\ell^*(G) := \frac{1}{\ell} \sum_{k=0}^{\ell-1} G\left(\frac{z+24k}{\ell}\right),
\mylabel{eq:Uellsdef}
\eeq
so that
$$
\alpha_\ell(z) = U_\ell^*(\alpha) - \chi_{12}(\ell) \, \ell \, \alpha(\ell z).
$$
The usual Atkin operator 
$U_\ell$ is defined by
\beq
U_\ell(G) := \frac{1}{\ell} \sum_{k=0}^{\ell-1} G\left(\frac{z+k}{\ell}\right).
\mylabel{eq:Uelldef}
\eeq
We need $U_\ell^*$ since $\alpha(z)$ has fractional powers of $q$, and we note that
$$
U_\ell^*(G) = U_\ell(G^*) (z/24),
$$
where $G^*(z) = G(24z)$. For a congruence subgroup $\Gamma$
we let $M_k(\Gamma)$ denote the space of entire modular
forms of weight $k$ with respect to the group $\Gamma$, 
and we let $M_k(\Gamma,\chi)$ denote the space of entire modular
forms of weight $k$ and character $\chi$ with respect to the group $\Gamma$.
Then

\begin{theorem}
\mylabel{thm:Gell}
If $\ell \ge5$ is prime, then
\beq
G_\ell(z) := \alpha_\ell(z) \, \frac{ \eta^{2\ell}(z)}{\eta(\ell z)} 
\in M_{\ell+1}(\Gamma_0(\ell)).
\mylabel{eq:Gelldef}
\eeq
In other words, the function $G_\ell(z)$ is an entire modular form of weight $\ell+1$
with respect to the group $\Gamma_0(\ell)$.
\end{theorem}
\begin{proof}
We assume $\ell\ge5$ is prime.
We divide the proof into four parts:
\begin{enumerate}
\item[(i)]
$U_\ell^*(\mathcal{A}) - \ell \, \chi_{12}(\ell) \, \mathcal{A}(\ell z)
=\alpha_\ell(z)$ and
$G_\ell(z)$ is holomorphic for $\Im(z)>0$.
\item[(ii)]
$G_\ell(Az) = (cz+d)^{\ell + 1} G_\ell(z)$ for all 
$A= \begin{pmatrix} a & b \\ c & d \end{pmatrix} \in \Gamma_0(\ell)$.
\item[(iii)]
$G_\ell(z)$ is holomorphic at $i\infty$.
\item[(iv)]
$G_\ell(z)$ is holomorphic at the cusp $0$.
\end{enumerate}

\noindent 
{Part (i)}. 
It is well-known (and an easy exercise) to show that
\beq
U_\ell(\eta(24z)) = \chi_{12}(\ell) \, \eta(24z).
\mylabel{eq:Uetatrans}
\eeq
Using \eqn{Nint} and \eqn{Uetatrans} we easily find that
$$
U_\ell(\mathcal{N}(z)) = \ell \, \chi_{12}(\ell) \mathcal{N}(z).
$$
It follows that
$$
U_\ell(\mathcal{M}) - \ell \, \chi_{12}(\ell) \, \mathcal{M}(\ell z)
$$
is holomorphic for $\Im(z)>0$.  By replacing $z$ by $\frac{z}{24}$ we see that
$$
U_\ell^*(\mathcal{A}) - \ell \, \chi_{12}(\ell) \, \mathcal{A}(\ell z)
= U_\ell^*(\alpha) - \ell \, \chi_{12}(\ell) \, \alpha(\ell z) = \alpha_\ell(z)
$$
and it is clear that $G_\ell(z)$ 
is holomorphic for $\Im(z)>0$.


\vskip 20pt

\noindent 
{Part (ii)}. 
Now let 
$A= \begin{pmatrix} a & b \\ c & d \end{pmatrix} \in \Gamma_0(\ell)$.
We must show that
$$
G_\ell(Az) = (cz+d)^{\ell + 1} G_\ell(z).
$$          
Since it is well-known that
$$
\left( \frac{\eta^\ell(z)}{\eta(\ell z)}\right)^2 
\in M_{\ell-1}(\Gamma_0(\ell)),
$$
it suffices to show that
$$
\alpha_\ell(Az) \, \eta(\ell Az) = (cz+d)^2 \alpha_\ell(z) \eta(\ell z).
$$
We need to show that

\begin{align}            
f_\ell(Az) &= (cz + d)^2 f_\ell(z), \mylabel{eq:felltrans}\\
g_\ell(Az) &= (cz + d)^2 g_\ell(z),  \mylabel{eq:gelltrans}
\end{align}            
where
$$
f_\ell(z) = U_\ell^*(\mathcal{A}) \, \eta(\ell z),\qquad 
g_\ell(z) = \mathcal{A}(\ell z) \, \eta(\ell z).               
$$
Let
$$
A^* =
\begin{pmatrix} a & \ell b \\ c/\ell  & d \end{pmatrix}.                                     
$$
Then $A^*\in\SL_2(\Z)$ and \eqn{gelltrans} follows from Theorem \thm{Atrans}
and the fact that
$$
\mathcal{A}(\ell Az) \, \eta(\ell Az) =              
\mathcal{A}(A^*z) \, \eta(A^* z).               
$$
Now,
$$
f_\ell(z) = U_\ell^*(\mathcal{A}) \eta(\ell z) = U_\ell^*(\mathcal{A}(z) \eta(\ell^2 z)).
$$
We define
\beq
F_\ell(z) := \mathcal{A}(z) \eta(\ell^2 z) 
= \mathcal{A}(z) \eta(z) \,    
\frac{\eta(\ell^2 z)}
{\eta(z)}.
\mylabel{eq:Felldef}
\eeq
Using Theorem \thm{Atrans} and the fact that $\frac{\eta(\ell^2 z)}{\eta(z)}$
is a modular function on $\Gamma_0(\ell^2)$ we have 
$$
F_\ell(Cz) = (c_1 z + d_1)^2 F_\ell(z),
$$
for  $C=
\begin{pmatrix}
a_1 & b_1 \\
c_1 & d_1
\end{pmatrix} \in \Gamma_0(\ell^2)$.

Now for $0 \le k \le\ell-1$, let
$$
B_k = \begin{pmatrix} 1 & 24k \\ 0 & \ell \end{pmatrix}
$$
so that
$$
f_\ell(z) = U_\ell^*(F_\ell(z)) = \frac{1}{\ell} \sum_{k=0}^{\ell-1} F_\ell(B_kz).
$$
Since $A\in\Gamma_0(\ell)$, $(a,\ell)=1$ and we can
choose unique $0 \le k^* \le \ell-1$ such that
$$
24ak^* \equiv b + 24kd \pmod{\ell}.
$$
Then
$$
B_k A = A_k^* B_{k^*},
$$
where $A_k^*\in\Gamma_0(\ell^2)$. 
%
%
We have
\beqs            
f_\ell(Az) =   \frac{1}{\ell} \sum_{k=0}^{\ell-1} F_\ell(B_k Az)
           =   \frac{1}{\ell} \sum_{k^*=0}^{\ell-1} F_\ell(A_k^* B_{k^*}z) 
           = \frac{ (cz+d)^2 }{\ell} \sum_{k^*=0}^{\ell-1} F_\ell( B_{k^*}z)
           = (cz+d)^2 f_\ell(z),
\eeqs            
which is \eqn{felltrans}.

\vskip 20pt

\noindent
{Part (iii)}.
First we note that $r_\ell^*$ is a positive integer.
We have
$$
G_\ell(z) = \alpha_\ell(z) \, \frac{ \eta^{2\ell}(z)}{\eta(\ell z)}
=
\sum_{n=-r_\ell^*}
\left(  \mathbf{a}(\ell n + d_\ell) - \chi_{12}(\ell) \, \ell \, \mathbf{a}\left ( \frac{n+r_\ell^*}{\ell}\right)
\right)
q^{n+ r_\ell*} \, \frac{E(q)^{2\ell}}{E(q^\ell)}
$$
where
$$
E(q) = \prod_{n=1}^\infty (1-q^n).
$$
We see that $G_\ell(z)$ is holomorphic at $i\infty$.

\vskip 20pt

\noindent 
{Part (iv)}. 
We need to find $G_\ell\left(\frac{-1}{\ell z}\right)$.
$$
U_\ell^*(\mathcal{A})  =   \frac{1}{\ell} \sum_{k=0}^{\ell -1} 
\mathcal{A}\left(\frac{z + 24k}{\ell}\right) 
=   \frac{1}{\ell}  \mathcal{A}\left(\frac{z}{\ell}\right) 
    +  
   \frac{1}{\ell} \sum_{k=1}^{\ell -1} 
\mathcal{A}\left(\frac{z + 24k}{\ell}\right)
=   \frac{1}{\ell}  \mathcal{A}\left(\frac{z}{\ell}\right) 
    +  
   \frac{1}{\ell} \sum_{k=1}^{\ell -1} 
\mathcal{A}(B_k z).                                
$$
For each $1 \le k \le \ell-1$ choose $1 \le k^* \le \ell-1$ such that 
$576 k k^* \equiv -1 \pmod{\ell}$.
Then
$$
B_k S = C_k B_{k^*},
$$
where
$$
C_k = \begin{pmatrix} 24k & \frac{-1-576kk^*}{\ell} \\ \ell & -24k^* \end{pmatrix}
\in \Gamma_0(\ell).
$$
Then
$$
\mathcal{A}(B_k S z) = \mathcal{A}(C_k B_{k^*}z) = z^{3/2} 
\leg{-24k^*}{\ell} e^{\pi i \ell/4} \mathcal{A}(B_{k^*}z),
$$
by Theorem \thm{Atrans} since
$$
\nu_\eta(C_k) = \leg{-24k^*}{\ell} e^{-\pi i \ell/4},
$$
by \cite[p.51]{Kn-book}. Define
$$
S_\ell = \begin{pmatrix} 0 & -1 \\ \ell & 0 \end{pmatrix}.
$$
By Theorem \thm{Atrans},
\begin{align*}
\mathcal{A}\left(\frac{1}{\ell} S_\ell z\right) &= e^{\pi i/4} (z \ell^2)^{3/2} 
\mathcal{A}(\ell^2 z),\\
\mathcal{A}\left(\ell S_\ell z\right) &= e^{\pi i/4} z^{3/2} 
\mathcal{A}(z).
\end{align*}.

Hence, if we define
\beq
H_\ell(z) := U_\ell^*(\mathcal{A}) - \ell \chi_{12} \mathcal{A}(\ell z),
\mylabel{eq:Helldef}
\eeq
then
$$
H_\ell(S_\ell z) = \ell z^{3/2} e^{\pi i/4}\left(
\mathcal{A}(\ell^2 z) + \frac{1}{\sqrt{\ell}} e^{\pi i(\ell -1)/4} \sum_{k=1}^{\ell -1}
\leg{-24k}{\ell}\mathcal{A}\left(z + \tfrac{24k}{\ell}\right) 
- \chi_{12}(\ell) \mathcal{A}(z) \right).
$$
Replacing $z$ by $24z$ gives
$$
H_\ell(S_\ell 24z) = \ell (24z)^{3/2} e^{\pi i/4}\left(
\mathcal{M}(\ell^2 z) + \frac{1}{\sqrt{\ell}} \chi_{12}(\ell) \epsilon_{\ell}^3 
\sum_{k=1}^{\ell -1}
\leg{-k}{\ell}\mathcal{M}\left(z + \tfrac{k}{\ell}\right) 
- \chi_{12}(\ell) \mathcal{M}(z) \right),
$$
since
$$
 e^{\pi i(\ell -1)/4} \leg{24}{\ell} = \chi_{12}(\ell) \epsilon_\ell^3.
$$
Here 
$$
\epsilon_{\ell} =
\begin{cases}
1 & \mbox{if $\ell\equiv1\pmod{4}$,}\\
i & \mbox{if $\ell\equiv3\pmod{4}$.}
\end{cases}
$$
By \cite[p.451]{Sh} we have
\begin{align*}
H_\ell(S_\ell 24z) 
&= \ell (24z)^{3/2} e^{\pi i/4}\left(
\mathcal{M} \vert T(\ell^2) - \chi_{12}(\ell) \mathcal{M}(z) - U_{\ell^2}(\mathcal{M})
\right),\\
&= \ell (24z)^{3/2} e^{\pi i/4}\left(
(\mathcal{M} \vert T(\ell^2) - \chi_{12}(\ell)(1 +\ell)\mathcal{M}(z)) - 
(U_{\ell^2}(\mathcal{M}) - \ell \chi_{12}(\ell) \mathcal{M}(z))
\right),
\end{align*}
where $T(\ell^2)$ is the Hecke operator which acts on harmonic Maass forms of
weight $\tfrac{3}{2}$, 
and was used by Ono \cite{On10}. When the form is meromorphic
it corresponds to the usual Hecke operator as described by Shimura \cite{Sh}.
Ono \cite{On10} showed that function
$$
\mathcal{M}_\ell(z) = \mathcal{M} \vert T(\ell^2) - \chi_{12}(\ell)(1 +\ell)\mathcal{M}(z))
$$
is a weakly holomorphic modular form. In fact, he showed that 
\beq
\mathcal{F}_\ell(z) := \eta(z)^{\ell^2} \mathcal{M}_\ell(z/24)
\mylabel{eq:Facdef}
\eeq
is a weight $(l^2+3)/2$ entire modular form on $\SL_2(\Z)$. See 
\cite[Theorem 2.2]{On10}. We also note that the function
$$
U_{\ell^2}(\mathcal{M}) - \ell \chi_{12}(\ell) \mathcal{M}(z)
= U_\ell\left( U_{\ell}(\mathcal{M}) - \ell \chi_{12}(\ell) \mathcal{M}(\ell z)\right)
$$
is holomorphic for $\Im(z)>0$ by the remarks in Part (i).
Thus we find that
\beq
G_\ell\left(\frac{-1}{\ell z}\right)   
= -(iz\ell)^{\ell + 1} \frac{E(q^\ell)^{2\ell}}{E(q)}
\left(
\sum_{n=-s_\ell}^\infty 
 \left( 
         \chi_{12}(\ell)  \mathbf{a}(n) \left(\leg{1-24n}{\ell}-1\right)
         + \ell \mathbf{a}\left ( \frac{n+s_\ell}{\ell^2}\right) \right)
                             q^{n + 2s_\ell}
\right),
\mylabel{eq:Gelltrans}
\eeq
where $s_\ell = \frac{\ell^2-1}{24}$.
It follows that $G_\ell(z)$ is holomorphic at the cusp $0$.
\end{proof}

Since $G_\ell(z)\in M_{\ell + 1}(\Gamma_0(\ell))$, the
function 
$z^{-\ell-1}  G_\ell\left(\frac{-1}{\ell z}\right)
\in M_{\ell + 1}(\Gamma_0(\ell))$  by \cite[Lemma 1]{At-Le}.
Thus if we define
\beq
\beta_\ell(z) :=
\sum_{n=-s_\ell}^\infty 
 \left( 
         \chi_{12}(\ell)  \mathbf{a}(n) \left(\leg{1-24n}{\ell}-1\right)
         + \ell \mathbf{a}\left ( \frac{n+s_\ell}{\ell^2}\right) \right)
                             q^{n  - \tfrac{1}{24}},
\mylabel{eq:betadef}
\eeq
then the proof of Part (iv) of  Theorem \thm{Gell} yields 
\begin{corollary}
\mylabel{cor:Jell}
If $\ell \ge5$ is prime, then
\beq
J_\ell(z) := \beta_\ell(z) \, \frac{ \eta^{2\ell}(\ell z)}{\eta(z)} \in M_{\ell+1}(\ell).
\mylabel{eq:Jelldef}
\eeq
\end{corollary}

We illustrate the case $\ell=5$. For $\ell$ prime we define
\beq
\mathcal{E}_{2,\ell}(z) 
:= \frac{1}{\ell -1}\left( \ell E_2(\ell z) - E_2(z) \right).
\mylabel{eq:E2elldef}
\eeq
It is well-known that $\mathcal{E}_{2,\ell}(z)\in M_{2}(\Gamma_0(\ell))$.
By \cite[Theorem 3.8]{Ki-book} $\dim  M_{6}(\Gamma_0(5)) = 3$, and it can be 
shown that
$$
\{
\mathcal{E}_{2,5}(z) \frac{\eta(5z)^{10}}{\eta(z)^2},
\mathcal{E}_{2,5}(z) \eta(5z)^{4} \eta(z)^4,
\mathcal{E}_{2,5}(z) \frac{\eta(z)^{10}}{\eta(5z)^2}
\}
$$
is a basis. We find that
$$
G_5(z) = 5 \, \mathcal{E}_{2,5}(z) \left(125 \, \eta(5z)^{4} \eta(z)^4
-   \frac{\eta(z)^{10}}{\eta(5z)^2}\right),
$$
and
$$
J_5(z) = 5 \, \mathcal{E}_{2,5}(z) \left( \frac{\eta(5z)^{10}}{\eta(z)^2}
-  \eta(5z)^{4} \eta(z)^4\right).
$$
Thus
\beq
\sum_{n=0}^\infty
\left(  \mathbf{a}(5 n - 1)  + 5\, \mathbf{a}\left ( \frac{n}{5}\right)\right)
q^{n- \tfrac{5}{24}} 
=
5 \, \frac{ \mathcal{E}_{2,5}(z) }{\eta(5z)}\left( 
125\,\frac{\eta(5z)^{6}}{\eta(z)^6} - 1\right),
\mylabel{eq:a5n1id}
\eeq
and
\beq
\sum_{n=-1}^\infty 
 \left( 
         -  \mathbf{a}(n) \left(\leg{1-24n}{5}-1\right)
         + 5 \, \mathbf{a}\left ( \frac{n+1}{25}\right) \right)
                             q^{n  - \tfrac{1}{24}} 
=
5 \, \frac{ \mathcal{E}_{2,5}(z) }{\eta(z)}\left( 
1 - \frac{\eta(z)^{6}}{\eta(5z)^6}\right).
\mylabel{eq:a5twistid}
\eeq

\section{The Congruences} \mylabel{sec:congs}

In this section we derive explicit formulas for the generating functions
of
\beq
\mathbf{a}(\ell^{a} n  + d_{\ell,a} )  - \chi_{12}(\ell)\, \ell \,
\mathbf{a}(\ell^{a-2} n + d_{\ell,a-2}),    
\mylabel{eq:genell}
\eeq
when $\ell=5$, $7$, and $13$. Here $24 d_{\ell,a} \equiv 1 \pmod{\ell^a}$.
The presentation of the identities is analogous to those of the
partition function as given by Hirschhorn and Hunt \cite{Hi-Hu}
and the author \cite{Ga84}. In each case we start 
by using Theorem \thm{Gell} to find identities for $\alpha_{\ell}(z)$.
This basically gives the initial case $a=1$. Then we use 
Watson's \cite{Wa38} and Atkin's \cite{At68} method of
modular equations to do the induction step and study the arithmetic
properties of the coefficients in these identities. The main congruences
\eqn{spt5acong}\--\eqn{spt13ccong} then follow in a straightforward way.

\subsection{The SPT-function modulo powers of $5$} \mylabel{subsec:5}

\begin{theorem}
\mylabel{thm:mainthm5}
If $a\ge1$ then
\begin{align}
\sum_{n=0}^\infty
\left(  \mathbf{a}(5^{2a-1} n  - t_a )  + 5\,  \mathbf{a}(5^{2a-3} n - t_{a-1})\right)
q^{n- \tfrac{5}{24}} 
&=
\frac{ \mathcal{E}_{2,5}(z) }{\eta(5z)} \sum_{i\ge0} x_{2a-1,i} Y^i, 
\mylabel{eq:gen5odd}\\
\sum_{n=0}^\infty
\left(  \mathbf{a}(5^{2a} n  - t_a )  + 5\,  \mathbf{a}(5^{2a-2} n - t_{a-1})\right)
q^{n- \tfrac{1}{24}} 
&=
\frac{ \mathcal{E}_{2,5}(z) }{\eta(z)} \sum_{i\ge0} x_{2a,i} Y^i, 
\mylabel{eq:gen5even}
\end{align}
where
$$
t_a=\frac{1}{24}(5^{2a}-1),\qquad Y(z) = \frac{\eta(5z)^6}{\eta(z)^6},
$$
$$
\vec{x}_1 = (x_{1,0},x_{1,1},\cdots) = (-5, 5^4, 0, 0, 0, \cdots),
$$
and for $a\ge1$
\beq
\vec{x}_{a+1} =
\begin{cases}
\vec{x}_a A, & \mbox{$a$ odd},\\
\vec{x}_a B, & \mbox{$a$ even}.   
\end{cases}
\mylabel{eq:xrec}
\eeq
Here  $A=(a_{i,j})_{i\ge 0, j\ge 0}$ and
      $B=(a_{i,j})_{i\ge 0, j\ge 0}$ are defined by
\beq
a_{i,j} = m_{6i,i+j},\qquad
b_{i,j} = m_{6i+1,i+j},
\mylabel{eq:abdefs}
\eeq
where the matrix $M=(m_{i,j})_{i,j\ge0}$ is defined as follows:
The first five rows of $M$ are
$$
\begin{pmatrix}
1     &     0     &     0     &     0     &     0     &     0 & \cdots\\
0     &     5^{3} &     0     &     0     &     0     &     0 & \cdots\\ 
0     & 4\mytimes5^{2} &     5^{5} &     0     &     0     &     0 & \cdots\\
0     & 9\mytimes5     & 9\mytimes5^{4} &     5^{7} &     0     &     0 & \cdots\\
0     & 2\mytimes5     &44\mytimes5^{3} &14\mytimes5^{6} &     5^{9} &     0 & \cdots 
\end{pmatrix}
$$
and for $i\ge5$, $m_{i,0}=0$ and for $j\ge 1$,
\beq
m_{i,j} = 25\,m_{i-1,j-1} + 25\,m_{i-2,j-1} + 15\,m_{i-3,j-1} + 5\,m_{i-4,j-1}
          + m_{i-5,j-1}.
\mylabel{eq:mrec5}
\eeq
\end{theorem}

\begin{lemma}
\mylabel{lem:5.2}
If $n$ is a positive integer then there are integers $c_m$ 
$(\lceil\frac{n}{5}\rceil \le m \le n)$
such that
$$
U_5(\mathcal{E}_{2,5} Z^n) 
= \mathcal{E}_{2,5} \sum_{m=\lceil\frac{n}{5}\rceil}^n c_m Y^m,
$$
where
\beq
Z(z) = \frac{\eta(25z)}{\eta(z)},\qquad Y(z) = \frac{\eta(5z)^6}{\eta(z)^6}.
\mylabel{eq:ZYdefs}
\eeq
\end{lemma}
\begin{proof}
We need the following dimension formulas which
follow from \cite{Co-Oe} and \cite[Theorem 3.8]{Ki-book}.
For $k$ even,
\begin{align*}
\dim M_k(\Gamma_0(5)) &= 2 \left\lfloor \frac{k}{4} \right\rfloor + 1,\\
\dim M_k(\Gamma_0(5),\leg{\cdot}{5}) &= k - 2 \left\lfloor \frac{k}{4} \right\rfloor.
\end{align*}
Let $n$ be a positive integer. Then
$$
U_5(\mathcal{E}_{2,5} Z^n) 
= U_5\left(\mathcal{E}_{2,5}(z) \left(\frac{\eta(5z)^5}{\eta(z)}\right)^n
                             \left(\frac{\eta(25z)}{\eta(5z)^5}\right)^n
\right)
= U_5\left(\mathcal{E}_{2,5}(z) \left(\frac{\eta(5z)^5}{\eta(z)}\right)^n
\right)  \left(\frac{\eta(5z)}{\eta(z)^5}\right)^n.
$$
When $n$ is even the function
$$
\mathcal{E}_{2,5}(z) \left(\frac{\eta(5z)^5}{\eta(z)}\right)^n
$$
belongs to the space $M_{2n+2}(\Gamma_0(5))$, which has as a basis
$$
\left\{\mathcal{E}_{2,5}(z) \eta(z)^{5n-6m} \eta(5z)^{6m-n}\:,\: 0 \le m \le n
\right\}. 
$$
This follows from the dimension formula. We note that
$$
\ord(\mathcal{E}_{2,5}(z) \eta(z)^{5n-6m} \eta(5z)^{6m-n}; i\infty)=m.
$$ 
The operator $U_5$ preserves the space $M_{2n+2}(\Gamma_0(5))$.
It follows that there are integers $c_m$ $(\lceil\frac{n}{5}\rceil \le m \le n$) such 
that
$$
U_5(\mathcal{E}_{2,5} Z^n) 
= \mathcal{E}_{2,5}(z) \sum_{m=\lceil\frac{n}{5}\rceil}^n c_m \,\eta(z)^{5n-6m} \eta(5z)^{6m-n}
 \left(\frac{\eta(5z)}{\eta(z)^5}\right)^n
= 
 \mathcal{E}_{2,5}(z) \sum_{m=\lceil\frac{n}{5}\rceil}^n c_m Y^m.
$$
When $n$ is odd the proof is similar except this time one needs
to work in the space
$M_{2n+2}(\Gamma_0(5),\leg{\cdot}{5})$.
\end{proof}

\begin{corollary}
\mylabel{cor:5.3}
\begin{align}
U_5(\mathcal{E}_{2,5}) &= \mathcal{E}_{2,5}\mylabel{eq:U5Z0}\\
U_5(\mathcal{E}_{2,5} Z) &= 5^3 \, \mathcal{E}_{2,5}  Y\mylabel{eq:U5Z1}\\
U_5(\mathcal{E}_{2,5} Z^2) &= 5^2 \, \mathcal{E}_{2,5}(4 Y + 5^3 Y^2)\mylabel{eq:U5Z2}\\
U_5(\mathcal{E}_{2,5} Z^3) 
&= 5 \, \mathcal{E}_{2,5}(9 Y + 9 \cdot 5^3  Y^2 + 5^6  Y^3)
\mylabel{eq:U5Z3}\\
U_5(\mathcal{E}_{2,5} Z^4) 
&= 5\,\mathcal{E}_{2,5} (2 Y + 44 \cdot 5^2 Y^2 + 14 \cdot 5^5 Y^3 
        + 5^8 Y^4).
\mylabel{eq:U5Z4}
\end{align}
\end{corollary}
\begin{proof}
Equation \eqn{U5Z0} is elementary. It also follows from the
fact that $\dim M_2(\Gamma_0(5)) = 1$.
Equations \eqn{U5Z1}--\eqn{U5Z4} follow from Lemma \lem{5.2} and
straightforward calculation.
\end{proof}

We need the 5th order modular equation that was used by Watson
to prove Ramanujan's partition congruences for powers of $5$.
\beq
Z^5 = \left(25 Z^4 + 25 Z^3 + 15 Z^2 + 5 Z + 1\right) Y(5z).
\mylabel{eq:ME5}
\eeq

\begin{lemma}
\mylabel{lem:5.4}
For $i\ge0$
$$
U_5( \mathcal{E}_{2,5} Z^i) 
=  \mathcal{E}_{2,5}(z) \,\sum_{j=\lceil\frac{i}{5}\rceil}^i m_{i,j} Y^j,
$$
where $Z=Z(z)$, $Y=Y(z)$ are defined in \eqn{ZYdefs}, and
the $m_{i,j}$ are defined in Theorem \thm{mainthm5}.
\end{lemma}
\begin{proof}
The result holds for $0 \le i \le 4$ by Corollary \cor{5.3}.
By \eqn{ME5} we have
$$
U_5( \mathcal{E}_{2,5} Z^i) = 
\left(
 25 U_5( \mathcal{E}_{2,5} Z^{i-1}) 
 + 25 U_5( \mathcal{E}_{2,5} Z^{i-2}) 
 + 15 U_5( \mathcal{E}_{2,5} Z^{i-3}) 
 + 5 U_5( \mathcal{E}_{2,5} Z^{i-4}) 
 + U_5( \mathcal{E}_{2,5} Z^{i-5}) \right) Y(z),
$$
for $i\ge 5$. The result follows by induction on $i$ using the
recurrence \eqn{mrec5}.
\end{proof}

\begin{lemma}
\mylabel{lem:5.5} For $i\ge0$,
\begin{align}
U_5( \mathcal{E}_{2,5} Y^i) 
& = \mathcal{E}_{2,5}(z)\,\sum_{j=\lceil\frac{i}{5}\rceil}^{5i} a_{i,j} Y^j,
\mylabel{eq:U5E2Yi}\\
U_5( \mathcal{E}_{2,5} Z Y^i) 
& = \mathcal{E}_{2,5}(z)\,\sum_{j=\lceil\frac{i+1}{5}\rceil}^{5i+1} b_{i,j} Y^j,
\mylabel{eq:U5E2ZYi}
\end{align}
where the $a_{i,j}$, $b_{i,j}$ are defined in \eqn{abdefs}.
\end{lemma}
\begin{proof}
Suppose $i\ge0$. By Lemma \lem{5.4}
\begin{align*}
U_5( \mathcal{E}_{2,5} Y^i) &= U_5( \mathcal{E}_{2,5} Z^{6i} Y(5z)^{-i})
= Y^{-i}  U_5( \mathcal{E}_{2,5} Z^{6i})\\
&= Y^{-i}  \mathcal{E}_{2,5}(z)\,\sum_{j=\lceil\frac{6i}{5}\rceil}^{6i} m_{6i,j} Y^j\\
&=   \mathcal{E}_{2,5}(z)\,\sum_{j\ge \lceil\frac{i}{5}\rceil}^{5i} m_{6i,i+j} Y^j 
=   \mathcal{E}_{2,5}(z)\,\sum_{j\ge \lceil\frac{i}{5}\rceil}^{5i} a_{i,j} Y^j,
\end{align*}
which is \eqn{U5E2Yi}.
Similarly
\begin{align*}
U_5( \mathcal{E}_{2,5} Z Y^i) &= U_5( \mathcal{E}_{2,5} Z^{6i+1} Y(5z)^{-i})
= Y^{-i} U_5( \mathcal{E}_{2,5} Z^{6i+1})\\
&= Y^{-i}  \mathcal{E}_{2,5}(z) \,\sum_{j=\lceil\frac{6i+1}{5}\rceil}^{6i+1}
                                    m_{6i+1,j} Y^j\\
&=  \mathcal{E}_{2,5}(z)\, \sum_{j=\lceil\frac{i+1}{5}\rceil}^{5i+1}
                                    m_{6i+1,i+j} Y^j 
=   \mathcal{E}_{2,5}(z)\,\sum_{j=\lceil\frac{i+1}{5}\rceil}^{5i+1} b_{i,j} Y^j,
\end{align*}
which is \eqn{U5E2ZYi}.
\end{proof}

\begin{proof}[Proof of Theorem \thm{mainthm5}]
We proceed by induction. The case $a=1$ of \eqn{gen5odd} is \eqn{a5n1id}.
We now suppose $a\ge1$ is fixed and \eqn{gen5odd} holds. Thus
$$                     
E(q^5) \,\sum_{n=0}^\infty
\left(  \mathbf{a}(5^{2a-1} n  - t_a )  + 5\,  \mathbf{a}(5^{2a-3} n - t_{a-1})\right)
q^{n}  
=
\mathcal{E}_{2,5}(z) \sum_{i\ge0} x_{2a-1,i} Y^i.
$$
We now apply the $U_5$ operator to both sides and use Lemma \lem{5.5}.
\begin{align*}         
& E(q) \, \sum_{n=0}^\infty
\left(  \mathbf{a}(5^{2a} n  - t_a )  + 5\,  \mathbf{a}(5^{2a-2} n - t_{a-1})\right)
q^{n}  
=
\sum_{i\ge0} x_{2a-1,i} U_5(\mathcal{E}_{2,5}(z) Y^i) \\
&= \mathcal{E}_{2,5}(z) 
\sum_{i\ge0} x_{2a-1,i} \sum_{j\ge0} a_{i,j} Y^j  
=
 \mathcal{E}_{2,5}(z)  \sum_{j\ge0}\left(\sum_{i\ge0} x_{2a-1,i} a_{i,j} \right)Y^j  
=
\mathcal{E}_{2,5}(z) \sum_{j\ge0} x_{2a,j} Y^j.
\end{align*}
We obtain \eqn{gen5even} by dividing both sides by $\eta(z)$.

Now again suppose $a$ is fixed and \eqn{gen5even} holds.
Multiplying both sides by $\eta(25z)$ gives
$$                     
E(q^{25})\,\sum_{n=0}^\infty
\left(  \mathbf{a}(5^{2a} n  - t_a )  + 5\,  \mathbf{a}(5^{2a-2} n - t_{a-1})\right)
q^{n+1}  
=
\mathcal{E}_{2,5}(z) \sum_{i\ge0} x_{2a,i}  Z Y^i.
$$
We apply the $U_5$ operator to both sides.                                   
$$                   
E(q^5) \, \sum_{n=0}^\infty
\left(  \mathbf{a}(5^{2a} (5n-1)  - t_a )  + 5\,  \mathbf{a}(5^{2a-2} (5n-1) - t_{a-1})\right)
q^{n}  
=
\sum_{i\ge0} x_{2a,i}  U_5(\mathcal{E}_{2,5}(z) Z Y^i).  
$$
Using  Lemma \lem{5.5} and the fact that $t_{a+1} = 5^{2a} + t_a$ we have
\begin{align*}       
E(q^5) \,\sum_{n=0}^\infty
&\left(  \mathbf{a}(5^{2a+1}n  - t_{a+1} )  + 5\,  \mathbf{a}(5^{2a-1}n - t_{a})\right)
q^{n}   
=
\mathcal{E}_{2,5}(z) \sum_{i\ge0} x_{2a,i}  \sum_{j\ge0} b_{i,j} Y^j  \\
&= \mathcal{E}_{2,5}(z)
\sum_{j\ge0}\left(\sum_{i\ge0} x_{2a,i} b_{i,j} \right)Y^j  
= \mathcal{E}_{2,5}(z)
\sum_{j\ge0} x_{2a+1,j} Y^j.
\end{align*}
We obtain \eqn{gen5odd} with $a$ replaced by $a+1$ after dividing both sides
by $\eta(5z)$. This completes the proof of the theorem.
\end{proof}

Throughout this section we will make repeated use of the following
lemma which we leave as an exercise.

\begin{lemma}
\mylabel{lem:floorbnd}
Suppose $x$, $y$, $n\in\Z$ and $n>0$. Then
\beq
\left\lfloor \frac{x}{n} \right\rfloor + 
\left\lfloor \frac{y}{n} \right\rfloor
 \ge 
\left\lfloor \frac{x+y-n+1}{n} \right\rfloor.
\mylabel{eq:floorineq}
\eeq
\end{lemma}

For any prime $\ell$ we let $\pi(n)=\pi_\ell(n)$ denote the
exact power of $\ell$ that divides $n$.

Then
\begin{lemma}
\mylabel{lem:5.6}
$$
\pi_5(m_{i,j}) \ge \lfloor \tfrac{1}{2}(5j - i +1)\rfloor,
$$
where the matrix $M=(m_{i,j})_{i,j\ge0}$ is defined in
Theorem \thm{mainthm5}.
\end{lemma}
\begin{proof}
First we verify the result for $0\le i \le 4$.
The result is easily proven for $i\ge 5$ using the 
recurrence \eqn{mrec5}. 
\end{proof}

\begin{corollary}
\mylabel{cor:5.7}
$$
\pi_5(a_{i,j}) \ge  \lfloor \tfrac{1}{2}(5j - i +1)\rfloor, \qquad
\pi_5(b_{i,j}) \ge  \lfloor \tfrac{1}{2}(5j - i)\rfloor, 
$$
where the $a_{i,j}$, $b_{i,j}$ are defined by \eqn{abdefs}.
\end{corollary}

\begin{lemma}
\mylabel{lem:5.8}
For $b\ge2$, and $j\ge1$,
\begin{align}
\pi_5(x_{2b-1,j}) &\ge 5b-6 +  \max(0,\lfloor \tfrac{1}{2}(5j - 7)\rfloor),
\mylabel{eq:nux2b1j}\\
\pi_5(x_{2b,j}) &\ge 5b-4 +  \lfloor \tfrac{1}{2}(5j - 5)\rfloor.
\mylabel{eq:nux2bj}
\end{align}
\end{lemma}
\begin{proof}
A calculation gives
\begin{align*}
\vec{x}_3 &= (x_{3,0}, x_{3,1}, x_{3,2}, \cdots )\\
&=(0, 669303124 \cdot 5^{4}, 3328977476 \cdot 5^{11}, 366098988268 \cdot 5^{14}, 
 201318006648837 \cdot 5^{15}, 1618593700646527 \cdot 5^{18}, \\
& 6370852555263938 \cdot 5^{21}, 
 2900024541422883 \cdot 5^{25}, 4237895677971369 \cdot 5^{28}, 
 21327793208615511 \cdot 5^{30}, \\
& 15532659183030861 \cdot 5^{33},
  8481639849706179 \cdot 5^{36}, 3564573506915806 \cdot 5^{39}, 
 1175454967692313 \cdot 5^{42}, \\
& 1542192101361916 \cdot 5^{44}, 
 325171329708596 \cdot 5^{47}, 55431641829564 \cdot 5^{50}, 
 1532152033009 \cdot 5^{54}, 171561318777 \cdot 5^{57}, \\
& 77490966671 \cdot 5^{59}, 5598792206 \cdot 5^{62}, 
 318906274 \cdot 5^{65}, 2799863 \cdot 5^{69}, \\
&91379 \cdot 5^{72}, 10439 \cdot 5^{74}, 149 \cdot 5^{77}, 5^{80}, 0, \cdots),
\\
\pi_5(\vec{x}_3) &= 
(\infty,4,11,14,15,18,21,25,28,30,33,36,39,42,44,47,50,54,57,59,62,65,69,72,74,77,80,
 \infty, \infty, \cdots),
\end{align*}
and \eqn{nux2b1j} holds for $b=2$. 
Now suppose $b\ge2$ is fixed and \eqn{nux2b1j}
holds. By \eqn{xrec}
$$
x_{2b,j} = \sum_{i\ge 1} x_{2b-1,i} a_{i,j}.
$$
Then using Corollary \cor{5.7} 
$$
\pi_5(x_{2b,1}) \ge \min( \{5b-4\} \cup \{ 
5b-6 +  \lfloor \tfrac{1}{2}(5i - 7)\rfloor + \lfloor(\tfrac{1}{2}(6-i)\rfloor\,:\, 2 \le i
\le 5\}) = 5b-4,
$$
and \eqn{nux2bj} holds for $j=1$. Suppose $j\ge 2$. Then
\begin{align*}
\pi_5(x_{2b,j}) &\ge  \min_{1 \le i\le 5j} (\pi_5(x_{2b-1,i}) + \pi_5(a_{i,j}))\\
&\ge  \min_{2 \le i\le 5j} (\pi_5(x_{2b-1,1}) + \pi_5(a_{1,j}),  
                     (\pi_5(x_{2b-1,i}) + \pi_5(a_{i,j}))\\
&\ge \min( \{5b-6 + \lfloor\tfrac{1}{2}(5j)\rfloor\}
     \cup
     \{
    5b-6 +  \lfloor \tfrac{1}{2}(5i - 7)\rfloor) + 
      \lfloor \tfrac{1}{2}(5j - i +1)\rfloor\,:\, 2\le i \le 5j\}).
\end{align*}
Now
$$
5b-6 + \lfloor\tfrac{1}{2}(5j)\rfloor
= 5b-4 + \lfloor\tfrac{1}{2}(5j-4)\rfloor.
$$
If $2 \le i \le 5j$, then using Lemma \lem{floorbnd} we have
\begin{align*}
5b-6 +  \lfloor \tfrac{1}{2}(5i - 7)\rfloor) + 
        \lfloor \tfrac{1}{2}(5j - i +1)\rfloor 
&\ge     5b-6 +  \lfloor \tfrac{1}{2}(5j + 4i -7)\rfloor \\
&\ge     5b-6 +  \lfloor \tfrac{1}{2}(5j + 1) \rfloor 
    =    5b-4 +  \lfloor \tfrac{1}{2}(5j - 3) \rfloor 
\end{align*}
and \eqn{nux2bj} holds.
Now suppose $b\ge2$ is fixed and \eqn{nux2bj}
holds. By \eqn{xrec}
$$
x_{2b+1,j} = \sum_{i\ge 1} x_{2b,i} b_{i,j}.
$$
We observe that $\pi_5(b_{1,1}) = \pi_5(500)=3$.
Then using Corollary \cor{5.7} 
$$
\pi_5(x_{2b+1,1}) \ge \min( \{5b-1\} \cup \{ 
5b-4 +  \lfloor \tfrac{1}{2}(5i - 4) + \lfloor(\tfrac{1}{2}(5-i)\rfloor\,:\, 2 \le i
\le 4\}) = 5b-1,
$$
and \eqn{nux2b1j} holds for $j=1$ with $b$ replaced by $b+1$. Suppose $j\ge 2$. Then
\begin{align*}
\pi_5(x_{2b+1,j}) &\ge  \min_{1 \le i\le 5j-1} (\pi_5(x_{2b,i}) + \pi_5(b_{i,j}))\\
&\ge  \min_{2 \le i\le 5j-1} (\pi_5(x_{2b,1}) + \pi_5(b_{1,j}),  
                     (\pi_5(x_{2b,i}) + \pi_5(b_{i,j}))\\
&\ge \min( \{5b-4 + \lfloor\tfrac{1}{2}(5j-1)\rfloor\}
     \cup
     \{
    5b-4 +  \lfloor \tfrac{1}{2}(5i - 4)\rfloor) + 
      \lfloor \tfrac{1}{2}(5j - i)\rfloor\,:\, 2\le i \le 5j-1\}).
\end{align*}
Now
$$
5b-4 + \lfloor\tfrac{1}{2}(5j-1)\rfloor
= 5b-1 + \lfloor\tfrac{1}{2}(5j-7)\rfloor.
$$
If $2 \le i \le 5j-1$, then again using Lemma \lem{floorbnd} we have
\begin{align*}
5b-4 +  \lfloor \tfrac{1}{2}(5i - 4)\rfloor) + 
        \lfloor \tfrac{1}{2}(5j - i)\rfloor
& \ge 5b-4 + \lfloor \tfrac{1}{2}(5j + 4i - 5) \rfloor \\
& \ge 5b-4 + \lfloor \tfrac{1}{2}(5j + 3) \rfloor    
=     5b-1 + \lfloor \tfrac{1}{2}(5j - 3) \rfloor    
\end{align*}
and \eqn{nux2b1j} holds with $b$ replaced by $b+1$.
Lemma \lem{5.8} follows by induction.
\end{proof}

\begin{corollary}
\label{cor:5.9}
For $b\ge 2$,
\begin{align}
 \mathbf{a}(5^{2b-1}n + \delta_{2b+1}) + 5 \,  \mathbf{a}(5^{2b-3}n + \delta_{2b-3}) &\equiv 0
\pmod{5^{5b-6}},
\mylabel{eq:a52b1cong}\\
 \mathbf{a}(5^{2b}n + \delta_{2b}) + 5 \,  \mathbf{a}(5^{2b-2}n + \delta_{2b-2}) &\equiv 0
\pmod{5^{5b-4}}.
\mylabel{eq:a52bcong}
\end{align}
For $a\ge1$,
\begin{align}
\spt(5^{a+2}n + \delta_{a+2}) + 5 \, \spt(5^{a}n + \delta_{a}) &\equiv 0
\pmod{5^{2a+1}},
\mylabel{eq:spt5abigcong}\\
\spt(5^{a}n + \delta_{a}) &\equiv 0
\pmod{5^{\lfloor \frac{a+1}{2}\rfloor}}.
\mylabel{eq:spt5alittlecong}
\end{align}
\end{corollary}
\begin{proof}
The congruences \eqn{a52b1cong}--\eqn{a52bcong} follow from Theorem
\thm{mainthm5} and Lemma \lem{5.8}.
Let
$$
\mbox{dp}(n) = (24n-1) p(n).
$$
Then
\beq
\mbox{dp}(5^a n + \delta_a) \equiv 0 \pmod{5^{2a}},
\mylabel{eq:dpcong5}
\eeq
by \eqn{ptn5cong}.
The congruence \eqn{spt5abigcong} follows from \eqn{a52b1cong}--\eqn{a52bcong},
and \eqn{dpcong5}. Andrews' congruence \eqn{spt5cong} implies
that \eqn{spt5alittlecong} holds for $a=1$, $2$. The general result follows
by induction using \eqn{spt5abigcong}.
\end{proof}

We note that when $a=0$ there is a stronger congruence than \eqn{spt5abigcong}.
We prove that
\beq
\spt(25n-1) + 5\, \spt(n) \equiv 0 \pmod{25}.
\mylabel{eq:spt25cong}
\eeq
We have calculated
\begin{align*}
\vec{x}_2 &= (x_{2,0}, x_{2,1}, x_{2,2}, \cdots )\\
&=(- 5^{1}, 63 \cdot 5^{6}, 104 \cdot 5^{9}, 189 \cdot 5^{11}, 
24 \cdot 5^{14}, 5^{17}, 0, \cdots ).
\end{align*}
Thus
\begin{align}
&\sum_{n=0}^\infty
\left(  \mathbf{a}(25 n - 1)  + 5\,  \mathbf{a}(n)\right)
q^{n- \tfrac{1}{24}} 
\mylabel{eq:a25n1id}
\\
&=
5 \frac{ \mathcal{E}_{2,5}(z) }{\eta(z)}\left( 
-1 +  63 \cdot 5^{5} \frac{\eta^{6}(5z)}{\eta^{6}(z)}
+ 104 \cdot 5^{8}  \frac{\eta^{12}(5z)}{\eta^{12}(z)}
+ 189 \cdot 5^{10}  \frac{\eta^{18}(5z)}{\eta^{18}(z)}
+ 24 \cdot 5^{13}  \frac{\eta^{24}(5z)}{\eta^{24}(z)}
+ 5^{16}  \frac{\eta^{30}(5z)}{\eta^{30}(z)}\right),
\nonumber                  
\end{align}
and
$$
\sum_{n=0}^\infty
\left(  \mathbf{a}(25 n - 1)  + 5\,  \mathbf{a}(n)\right)
q^{n- \tfrac{1}{24}} 
\equiv 
20 \frac{E_2(z)}{\eta(z)} \pmod{25}.
$$
But from \eqn{dpid} we see that
$$
\sum_{n=0}^\infty
\left( dp(25 n - 1)  + 5\, dp(n)\right)
q^{n- \tfrac{1}{24}} 
\equiv 
20 \frac{E_2(z)}{\eta(z)} \pmod{25},
$$
and
\begin{align*}
&12 \sum_{n=0}^\infty
\left( \spt(25 n - 1)  + 5\, \spt(n)\right)
q^{n- \tfrac{1}{24}} \\
&=
\sum_{n=0}^\infty
\left(  \mathbf{a}(25 n - 1)  + 5\,  \mathbf{a}(n)\right)
q^{n- \tfrac{1}{24}} 
-
\sum_{n=0}^\infty
\left( dp(25 n - 1)  + 5\, dp(n)\right)
q^{n- \tfrac{1}{24}}  \equiv0\pmod{25},
\end{align*}
which gives \eqn{spt25cong}.


\subsection{The SPT-function modulo powers of $7$} 
\mylabel{subsec:7}
\begin{theorem}
\mylabel{thm:mainthm7}
If $a\ge1$ then
\begin{align}
\sum_{n=0}^\infty
\left(  \mathbf{a}(7^{2a-1} n  - u_a )  + 7\,  \mathbf{a}(7^{2a-3} n - u_{a-1})\right)
q^{n- \tfrac{7}{24}} 
&=
\frac{ \mathcal{E}_{2,7}(z) }{\eta(7z)} \sum_{i\ge0} x_{2a-1,i} Y^i, 
\mylabel{eq:gen7odd}\\
\sum_{n=0}^\infty
\left(  \mathbf{a}(7^{2a} n  - u_a )  + 7\,  \mathbf{a}(7^{2a-2} n - u_{a-1})\right)
q^{n- \tfrac{1}{24}} 
&=
\frac{ \mathcal{E}_{2,7}(z) }{\eta(z)} \sum_{i\ge0} x_{2a,i} Y^i, 
\mylabel{eq:gen7even}
\end{align}
where
$$
u_a=\frac{1}{24}(7^{2a}-1),\qquad Y(z) = \frac{\eta(7z)^4}{\eta(z)^4},
$$
$$
\vec{x}_1 = (x_{1,0},x_{1,1},\cdots) = (-7, 3\mytimes 7^3, 7^5, 0, 0, \cdots),
$$
and for $a\ge1$
$$
\vec{x}_{a+1} =
\begin{cases}
\vec{x}_a A, & \mbox{$a$ odd},\\
\vec{x}_a B, & \mbox{$a$ even}.   
\end{cases}
$$
Here  $A=(a_{i,j})_{i\ge 0, j\ge 0}$ and
      $B=(a_{i,j})_{i\ge 0, j\ge 0}$ are defined by
\beq
a_{i,j} = m_{4i,i+j},\qquad
b_{i,j} = m_{4i+1,i+j},
\mylabel{eq:ab7defs}
\eeq
where the matrix $M=(m_{i,j})_{i,j\ge0}$ is defined as follows:
The first seven rows of $M$ are defined so that
$$
U_7( \mathcal{E}_{2,7} Z^i) = \sum_{j=\lceil\frac{2i}{7}\rceil}^{2i} m_{i,j} Y^j
\qquad(0 \le i \le 6),
$$
where
$$
Z(z) = \frac{\eta(49z)}{\eta(z)}. 
$$
and for $i\ge7$, $m_{i,0}=0$, $m_{i,1}=0$, and for $j\ge 2$,
\begin{align}
m_{i,j} &= 49\,m_{i-1,j-1}+35\,m_{i-2,j-1}+7\,m_{i-3,j-1}+343\,m_{i-1,j-2}
	  +343\,m_{i-2,j-2}
	  + 147\,m_{i-3,j-2} 
\mylabel{eq:mrec7}\\
&\qquad + 49\,m_{i-4,j-2}+21\,m_{i-5,j-2}+7\,m_{i-6,j-2}
	  +m_{i-7,j-2}.
\nonumber             
\end{align}
\end{theorem}

The proof of the following lemma is analogous to that of
Lemma \lem{5.2}.

\begin{lemma}
\mylabel{lem:7.2}
If $n$ is a positive integer then there are integers $c_m$ 
($\lceil\frac{2n}{7}\rceil \le m \le 2n$)
such that
$$
U_7(\mathcal{E}_{2,7} Z^n) 
= \mathcal{E}_{2,7} \sum_{m=\lceil\frac{2n}{7}\rceil}^{2n} c_m Y^m,
$$
where
\beq
Z(z) = Z_7(z) = \frac{\eta(49z)}{\eta(z)},\qquad Y(z) = \frac{\eta(7z)^4}{\eta(z)^4}.
\mylabel{eq:ZY7defs}
\eeq
\end{lemma}

\begin{corollary}
\mylabel{cor:7.3}
\begin{align}
U_7(\mathcal{E}_{2,7}) & = 
\mathcal{E}_{2,7} 
\mylabel{eq:U7Z0}\\
U_7(\mathcal{E}_{2,7}Z) & = 7^2 \mathcal{E}_{2,7} (
3 \,Y  +7^{2} \,Y^{2})
\mylabel{eq:U7Z1}\\
U_7(\mathcal{E}_{2,7}Z^2) & =   7 \mathcal{E}_{2,7} (
10 \,Y  +27 \cdot 7^{2}  \,Y^{2} +10 \cdot 7^{4}  \,Y^{3} +7^{6} \,Y^{4})
\mylabel{eq:U7Z2}\\
U_7(\mathcal{E}_{2,7}Z^3) & =  7 \mathcal{E}_{2,7} (
Y  +190 \cdot 7  \,Y^{2} +255 \cdot 7^{3}  \,Y^{3} +104 \cdot 7^{5}  \,Y^{4} +17 \cdot 7^{7}  \,Y^{5} +7^{9} \,Y^{6})
\mylabel{eq:U7Z3}\\
U_7(\mathcal{E}_{2,7}Z^4) & =  7^2 \mathcal{E}_{2,7} (
82 \,Y^{2} +352 \cdot 7^{2}  \,Y^{3} +2535 \cdot 7^{3}  \,Y^{4} +1088 \cdot 7^{5}  \,Y^{5} +230 \cdot 7^{7}  \,Y^{6}
\mylabel{eq:U7Z4}\\
&\qquad +24 \cdot 7^{9}  \,Y^{7} +7^{11} \,Y^{8})
\nonumber\\
U_7(\mathcal{E}_{2,7}Z^5) & =   7 \mathcal{E}_{2,7} (
114 \,Y^{2} +253 \cdot 7^{3}  \,Y^{3} +4169 \cdot 7^{4}  \,Y^{4} +3699 \cdot 7^{6}  \,Y^{5} +11495 \cdot 7^{7}  \,Y^{6} 
\mylabel{eq:U7Z5}\\
&\qquad +2852 \cdot 7^{9}  \,Y^{7} +405 \cdot 7^{11}  \,Y^{8} +31 \cdot 7^{13}  \,Y^{9} +7^{15} \,Y^{10})
\nonumber\\
U_7(\mathcal{E}_{2,7}Z^6) & =  7 \mathcal{E}_{2,7} (
9 \,Y^{2} +736 \cdot 7^{2}  \,Y^{3} +27970 \cdot 7^{3}  \,Y^{4} +6808 \cdot 7^{6}  \,Y^{5} +38475 \cdot 7^{7}  \,Y^{6} 
\mylabel{eq:U7Z6}\\
&\qquad +17490 \cdot 7^{9}  \,Y^{7} +33930 \cdot 7^{10}  \,Y^{8} +5890 \cdot 7^{12}  \,Y^{9} +629 \cdot 7^{14}  \,Y^{10} 
\nonumber\\
&\qquad +38 \cdot 7^{16}  \,Y^{11} +7^{18} \,Y^{12})
\nonumber
\end{align}
\end{corollary}

We need the 7th order modular equation that was used by Watson
to prove Ramanujan's partition congruences for powers of $7$.
\beq
Z^7 = (1+7\,Z+21\,Z^2+49\,Z^3+147\,Z^4+343\,Z^5+343\,Z^6)\,Y(7z)^2+(7\,Z^4+35\,Z^5+49\,Z^6)\,Y(7z)
\mylabel{eq:ME7}
\eeq

\begin{lemma}
\mylabel{lem:7.4}
For $i\ge0$
$$
U_7( \mathcal{E}_{2,7} Z^i) =  
\mathcal{E}_{2,7}(z) \,\sum_{j=\lceil\frac{2i}{7}\rceil}^{2i}
 m_{i,j} Y^j,
$$
where $Z=Z(z)$, $Y=Y(z)$ are defined in \eqn{ZY7defs}, and
the $m_{i,j}$ are defined in Theorem \thm{mainthm7}.
\end{lemma}

\begin{lemma}
\mylabel{lem:7.5} For $i\ge0$,
\begin{align}
U_7( \mathcal{E}_{2,7} Y^i) 
&=  \mathcal{E}_{2,7}(z)\,\sum_{j=\lceil\frac{i}{7}\rceil}^{7i} a_{i,j} Y^j,
\mylabel{eq:U7E2Yi}\\
U_7( \mathcal{E}_{2,7} Z Y^i) 
&=  \mathcal{E}_{2,7}(z)\,\sum_{j=\lceil\frac{i+2}{7}\rceil}^{7i+2} b_{i,j} Y^j    
\mylabel{eq:U7E2ZYi}
\end{align}
where the $a_{i,j}$, $b_{i,j}$ are defined in \eqn{ab7defs}.
\end{lemma}

Let $\pi_7(n)$ denote the exact power of $7$ dividing $n$.
Then
\begin{lemma}
\mylabel{lem:7.6}
$$
\pi_7(m_{i,j}) \ge \lfloor \tfrac{1}{4}(7j - 2i +3)\rfloor,
$$
where the matrix $M=(m_{i,j})_{i,j\ge0}$ is defined in
Theorem \thm{mainthm7}.
\end{lemma}

\begin{corollary}
\mylabel{cor:7.7}
$$
\pi_7(a_{i,j}) \ge  \lfloor \tfrac{1}{4}(7j - i +3)\rfloor, \qquad
\pi_7(b_{i,j}) \ge  \lfloor \tfrac{1}{4}(7j - i + 1)\rfloor, \qquad
$$
where the $a_{i,j}$, $b_{i,j}$ are defined by \eqn{ab7defs}.
\end{corollary}

\begin{lemma}
\mylabel{lem:7.8}
For $b\ge2$, and $j\ge1$,
\begin{align}
\pi_7(x_{2b-1,j}) &\ge 3b-3 +  \lfloor \tfrac{1}{4}(7j - 4)\rfloor.
\mylabel{eq:nux72b1j}\\
\pi_7(x_{2b,j}) &\ge 3b-1  +  \lfloor \tfrac{1}{4}(7j - 6)\rfloor.
\mylabel{eq:nux72bj}
\end{align}
\end{lemma}


\begin{corollary}
\label{cor:7.9}
For $b\ge 2$,
\begin{align}
 \mathbf{a}(7^{2b-1}n + \lambda_{2b+1}) + 7 \cdot  \mathbf{a}(7^{2b-3}n + \lambda_{2b-3}) &\equiv 0
\pmod{7^{3b-3}},
\mylabel{eq:a72b1cong}\\
 \mathbf{a}(7^{2b}n + \lambda_{2b}) + 7 \cdot  \mathbf{a}(7^{2b-2}n + \lambda_{2b-2}) &\equiv 0
\pmod{7^{3b-1}}.
\mylabel{eq:a72bcong}
\end{align}
For $a\ge1$,
\begin{align}
\spt(7^{a+2}n + \lambda_{a+2}) + 7 \cdot \spt(7^{a}n + \lambda_{a}) &\equiv 0
\pmod{7^{\lfloor\frac{1}{2}(3a+4)\rfloor}},
\mylabel{eq:spt7abigcong}\\
\spt(7^{a}n + \lambda_{a}) &\equiv 0
\pmod{7^{\lfloor \frac{a+1}{2}\rfloor}}.
\mylabel{eq:spt7alittlecong}
\end{align}
\end{corollary}

We note that \eqn{spt7abigcong} also holds for $a=0$.
The proof of the congruence
\beq
\spt(49n-2) + 7\cdot \spt(n) \equiv 0 \pmod{49}.
\mylabel{eq:spt49cong}
\eeq
is analogous to the proof of \eqn{spt25cong}.

\subsection{The SPT-function modulo powers of $13$}

\begin{theorem}
\mylabel{thm:mainthm13}
If $a\ge1$ then
\begin{align}
\sum_{n=0}^\infty
\left(  \mathbf{a}(13^{2a-1} n  - v_a )  - 13\,  \mathbf{a}(13^{2a-3} n - v_{a-1})\right)
q^{n- \tfrac{13}{24}} 
&=
\frac{ \mathcal{E}_{2,13}(z) }{\eta(13z)} \sum_{i\ge0} x_{2a-1,i} Y^i, 
\mylabel{eq:gen13odd}\\
\sum_{n=0}^\infty
\left(  \mathbf{a}(13^{2a} n  - v_a )  - 13\,  \mathbf{a}(13^{2a-2} n - v_{a-1})\right)
q^{n- \tfrac{1}{24}} 
&=
\frac{ \mathcal{E}_{2,13}(z) }{\eta(z)} \sum_{i\ge0} x_{2a,i} Y^i, 
\mylabel{eq:gen13even}
\end{align}
where
$$
v_a=\frac{1}{24}(13^{2a}-1),\qquad Y(z) = \frac{\eta(13z)^2}{\eta(z)^2},
$$
$$
\vec{x}_1 = (x_{1,0},x_{1,1},\cdots) = 
(13,  11 \cdot 13^{2}, 108 \cdot 13^{3}, 190 \cdot 13^{4}, 140 \cdot 13^{5}, 54 \cdot 13^{6}, 11 \cdot 13^{7}, 13^{8},   0, 0, 0, \cdots),
$$
and for $a\ge1$
$$
\vec{x}_{a+1} =
\begin{cases}
\vec{x}_a A, & \mbox{$a$ odd},\\
\vec{x}_a B, & \mbox{$a$ even}.   
\end{cases}
$$
Here  $A=(a_{i,j})_{i\ge 0, j\ge 0}$ and
      $B=(a_{i,j})_{i\ge 0, j\ge 0}$ are defined by
\beq
a_{i,j} = m_{2i,i+j},\qquad
b_{i,j} = m_{2i+1,i+j},
\mylabel{eq:ab13defs}
\eeq
where the matrix $M=(m_{i,j})_{i\ge-12,j\ge-6}$ is defined as follows:
The first 13 rows of $M$ are
$$
\begin{pmatrix}
0 & 0 & 0 & 0 & 0 & 0 & 13^{6} & 0 & 0 & \cdots \\ 
0 & 82 \mytimes 13 &456 \mytimes 13^{2} &360 \mytimes 13^{3} &126 \mytimes 13^{4} &18 \mytimes 13^{5} &13^{6} & 0 & 0 & \cdots \\ 
0 & 0 & 0 & 0 & 0 & 0 & 13^{5} & 0 & 0 & \cdots \\ 
0 & 0 & 18 \mytimes 13 &-36 \mytimes 13^{2} &-40 \mytimes 13^{3} &-14 \mytimes 13^{4} &-13^{5} &0 & 0 & \cdots \\ 
0 & 0 & 0 & 0 & 0 & 0 & 13^{4} & 0 & 0 & \cdots \\ 
0 & 0 & 0 & -14 \mytimes 13 &-12 \mytimes 13^{2} &0 & 13^{4} & 0 & 0 & \cdots \\ 
0 & 0 & 0 & 0 & 0 & 0 & 13^{3} & 0 & 0 & \cdots \\ 
0 & 0 & 0 & 0 & 4 \mytimes 13 &6 \mytimes 13^{2} &13^{3} & 0 & 0 & \cdots \\ 
0 & 0 & 0 & 0 & 0 & 0 & 13^{2} & 0 & 0 & \cdots \\ 
0 & 0 & 0 & 0 & 0 & 0 & -13^{2} &0 & 0 & \cdots \\ 
0 & 0 & 0 & 0 & 0 & 0 & 13 & 0 & 0 & \cdots \\ 
0 & 0 & 0 & 0 & 0 & 0 & -13 &0 & 0 & \cdots \\ 
0 & 0 & 0 & 0 & 0 & 0 & 1 & 0 & 0 & \cdots 
\end{pmatrix}
$$
and for $m_{k,\ell}=0$ for $k \ge 1$ and $-6 \le \ell\le 0$; and for 
$i\ge 1$ and $j\ge 1$,
\beq
m_{i,j} = 
\sum_{r=1}^{13} \sum_{s=\lfloor\frac{1}{2}(r+2)\rfloor}^7
\psi_{r,s} m_{i-r,j-s},
\mylabel{eq:mrec13}
\eeq
where $\Psi=(\psi_{r,s})_{1 \le r \le 13, 1 \le s \le 7}$ is 
the matrix
\beq
\Psi=\begin{pmatrix}
11 \mytimes 13 &36 \mytimes 13^{2} &38 \mytimes 13^{3} &20 \mytimes 13^{4} &6 \mytimes 13^{5} &13^{6} & 13^{6}  \\ 
0 & -204 \mytimes 13 &-346 \mytimes 13^{2} &-222 \mytimes 13^{3} &-74 \mytimes 13^{4} &-13^{6} &-13^{6} \\ 
0 & 36 \mytimes 13 &126 \mytimes 13^{2} &102 \mytimes 13^{3} &38 \mytimes 13^{4} &7 \mytimes 13^{5} &7 \mytimes 13^{5} \\ 
0 & 0 & -346 \mytimes 13 &-422 \mytimes 13^{2} &-184 \mytimes 13^{3} &-37 \mytimes 13^{4} &-3 \mytimes 13^{5} \\ 
0 & 0 & 38 \mytimes 13 &102 \mytimes 13^{2} &56 \mytimes 13^{3} &13^{5} & 15 \mytimes 13^{4} \\ 
0 & 0 & 0 & -222 \mytimes 13 &-184 \mytimes 13^{2} &-51 \mytimes 13^{3} &-5 \mytimes 13^{4} \\ 
0 & 0 & 0 & 20 \mytimes 13 &38 \mytimes 13^{2} &13^{4} & 19 \mytimes 13^{3} \\ 
0 & 0 & 0 & 0 & -74 \mytimes 13 &-37 \mytimes 13^{2} &-5 \mytimes 13^{3} \\ 
0 & 0 & 0 & 0 & 6 \mytimes 13 &7 \mytimes 13^{2} &15 \mytimes 13^{2} \\ 
0 & 0 & 0 & 0 & 0 & -13^{2} &-3 \mytimes 13^{2} \\ 
0 & 0 & 0 & 0 & 0 & 13 & 7 \mytimes 13 \\ 
0 & 0 & 0 & 0 & 0 & 0 & -13 \\ 
0 & 0 & 0 & 0 & 0 & 0 & 1 
\end{pmatrix}.
\mylabel{eq:Psidef}
\eeq
\end{theorem}

The proof of the following lemma is analogous to that of
Lemma \lem{5.2}.

\begin{lemma}
\mylabel{lem:13.2}
If $n$ is a positive integer then there are integers $c_m$ 
$(\lceil\frac{7n}{13}\rceil \le m \le 7n)$
such that
$$
U_{13}(\mathcal{E}_{2,13} Z^n) 
= \mathcal{E}_{2,13} \sum_{m=\lceil\frac{7n}{13}\rceil}^{7n} c_m Y^m,
$$
where
\beq
Z(z) = Z_{13}(z) = \frac{\eta(169z)}{\eta(z)},\qquad Y(z) = \frac{\eta(13z)^2}{\eta(z)^2}.
\mylabel{eq:ZY13defs}
\eeq
\end{lemma}

We need a version for Lemma \lem{13.2} when $n$ is negative.

\begin{lemma}
\mylabel{lem:13.2a}
If $n$ is a nonnegative integer then there are integers $c_m$ 
$(-6n\le m \le n-\lceil\frac{6n}{13}\rceil)$                         
such that
$$
U_{13}(\mathcal{E}_{2,13} Z^{-n}) 
= \mathcal{E}_{2,13} \sum_{m=-6n}^{n-\lceil\frac{6n}{13}\rceil} c_m Y^{-m}.
$$
\end{lemma}
\begin{proof}
The proof is analogous to Lemma \lem{13.2}. The main difference is that we write
$$
U_{13}(\mathcal{E}_{2,13} Z^{-n}) 
= U_{13}\left(\mathcal{E}_{2,{13}}(z) 
\left(\eta(z) \eta^{11}({13}z)\right)^n \right)
\left(\eta^{11}(z) \eta({13}z)\right)^{-n},
$$
and use the fact that
$\mathcal{E}_{2,{13}}(z)
\left(\eta(z) \eta^{11}({13}z)\right)^n\in M_{2+6n}(\Gamma_{0}(13),\leg{\cdot}{13}^n)$.
\end{proof}

\begin{corollary}
\mylabel{cor:13.3}
\begin{align*}
U_{13}(\mathcal{E}_{2,13} ) & = \mathcal{E}_{2,13}
\\
U_{13}(\mathcal{E}_{2,13} Z^{-1}) & = -13 \, \mathcal{E}_{2,13}
\\
U_{13}(\mathcal{E}_{2,13} Z^{-2}) & = 13 \, \mathcal{E}_{2,13}
\\
U_{13}(\mathcal{E}_{2,13} Z^{-3}) & = -13^2 \, \mathcal{E}_{2,13}
\\
U_{13}(\mathcal{E}_{2,13} Z^{-4}) & = 13^2 \, \mathcal{E}_{2,13}
\\
U_{13}(\mathcal{E}_{2,13} Z^{-5}) & = 13 \, \mathcal{E}_{2,13}(
4 \,Y^{-2} + 6 \cdot 13\,Y^{-1} + 13^{2})\\
U_{13}(\mathcal{E}_{2,13} Z^{-6}) & = 13^3 \, \mathcal{E}_{2,13}
\\
U_{13}(\mathcal{E}_{2,13} Z^{-7}) & = 13 \, \mathcal{E}_{2,13}(
-14 \,Y^{-3} - 12 \cdot 13\,Y^{-2} + 13^{3})\\
U_{13}(\mathcal{E}_{2,13} Z^{-8}) & = 13^4 \, \mathcal{E}_{2,13}
\\
U_{13}(\mathcal{E}_{2,13} Z^{-9}) & = 13 \, \mathcal{E}_{2,13}(
18 \,Y^{-4} - 36 \cdot 13\,Y^{-3} - 40 \cdot 13^{2}\,Y^{-2} - 14 \cdot 13^{3}\,Y^{-1} - 13^4)\\
U_{13}(\mathcal{E}_{2,13} Z^{-10}) & = 13^5 \, \mathcal{E}_{2,13}
\\
U_{13}(\mathcal{E}_{2,13} Z^{-11}) & = 13 \, \mathcal{E}_{2,13}(
82 \,Y^{-5} + 456 \cdot 13\,Y^{-4} + 360 \cdot 13^{2}\,Y^{-3} + 126 \cdot 13^{3}\,Y^{-2} + 18 \cdot 13^{4}\,Y^{-1} + 13^{5})\\
U_{13}(\mathcal{E}_{2,13} Z^{-12}) & = 13^6 \, \mathcal{E}_{2,13}
\\
\end{align*}
\end{corollary}

We need the 13th order modular equation that was used by Atkin and
O'Brien \cite{At-OB} to study properties of $p(n)$ modulo powers of $13$.
Lehner \cite{Le49} derived this equation earlier.

\beq
Z^{13}(z) = 
\sum_{r=1}^{13} \sum_{s=\lfloor\frac{1}{2}(r+2)\rfloor}^7
\psi_{r,s} Y^{s}(13z) \, Z^{13-r}(z),
\mylabel{eq:ME13}
\eeq
where the matrix $\Psi=(\psi_{i,j})$ is given in \eqn{Psidef},
and $Y(z)$, $Z(z)$ are given in \eqn{ZY13defs}.
The modular equation and the matrix $\Psi$  are given explicitly in
Appendix C in \cite{At-OB}

\begin{lemma}
\mylabel{lem:13.4}
For $i\ge0$
$$
U_{13}( \mathcal{E}_{2,13} Z^i) =  
\mathcal{E}_{2,13}(z) \,\sum_{j=\lceil\frac{7i}{13}\rceil}^{7i}
 m_{i,j} Y^j,
$$
where $Z=Z(z)$, $Y=Y(z)$ are defined in \eqn{ZY13defs}, and
the $m_{i,j}$ are defined in Theorem \thm{mainthm13}.
\end{lemma}

\begin{lemma}
\mylabel{lem:13.5} For $i\ge0$,
\begin{align}
U_{13}( \mathcal{E}_{2,{13}} Y^i) 
&=\mathcal{E}_{2,{13}}(z)\,\sum_{j=\lceil\frac{i}{13}\rceil}^{13i} a_{i,j} Y^j,
\mylabel{eq:U{13}E2Yi}\\
U_{13}( \mathcal{E}_{2,{13}} Z Y^i) 
&=\mathcal{E}_{2,{13}}(z)\,\sum_{j=\lceil\frac{i+7}{13}\rceil}^{13i+7} b_{i,j} Y^j    
\mylabel{eq:U13E2ZYi}
\end{align}
where the $a_{i,j}$, $b_{i,j}$ are defined in \eqn{ab13defs}.
\end{lemma}

Let $\pi_{13}(n)$ denote the exact power of $13$ dividing $n$.
Then
\begin{lemma}
\mylabel{lem:13.6}
For $i$, $j\ge0$,
\beq
\pi_{13}(m_{i,j}) \ge \lfloor \tfrac{1}{14}(13j - 7i +13)\rfloor,
\mylabel{eq:nu13m}
\eeq
where the matrix $M=(m_{i,j})$ is defined in
Theorem \thm{mainthm13}.
\end{lemma}
\begin{proof}
As noted in \cite{At-OB} we observe that
\beq
\pi_{13}(\psi_{r,s}) \ge  \lfloor \tfrac{1}{14}(13s - 7r +13)\rfloor,
\mylabel{eq:nu13psi}
\eeq
for all $1 \le t \le 13$ and $1 \le s \le 13$.
We verify the result for $0 \le i \le 12$ by direct computation
using the recurrence \eqn{mrec13}. We use \eqn{nu13psi}, the recurrence
\eqn{mrec13} and Lemma \lem{floorbnd}
to prove the general result by induction.
\end{proof}

\begin{corollary}
\mylabel{cor:13.7}
$$
\pi_{13}(a_{i,j}) \ge  \lfloor \tfrac{1}{14}({13}j - i +13)\rfloor, \qquad
\pi_{13}(b_{i,j}) \ge  \lfloor \tfrac{1}{14}({13}j - i + 6)\rfloor, \qquad
$$
where the $a_{i,j}$, $b_{i,j}$ are defined by \eqn{ab13defs}.
\end{corollary}

We provide more complete details for the proof of the following lemma since
congruences for the spt-function modulo $13$ are stronger than those
for the partition function.
\begin{lemma}
\mylabel{lem:13.8}
\begin{align}
\pi_{13}(x_{2,0}) &=1, 
\mylabel{eq:nu13x20}\\
\pi_{13}(x_{2,j}) &\ge 3  +  \lfloor \tfrac{1}{14}(13j)\rfloor
\qquad\mbox{for $j\ge1$}
\mylabel{eq:nu13x2j}\\
\pi_{13}(x_{2b-1,j}) &\ge 2b-2 +  \lfloor \tfrac{1}{14}(13j - 10)\rfloor
\qquad\mbox{for $b\ge 2$, and $j\ge1$}
\mylabel{eq:nu13x2b1j}\\
\pi_{13}(x_{2b,j}) &\ge 2b-1 +  \lfloor \tfrac{1}{14}(13j)\rfloor
\qquad\mbox{for $b\ge 2$, and $j\ge1$}.      
\mylabel{eq:nu13x2bj}
\end{align}
\end{lemma}
\begin{proof}
We have calculated $\vec{x}_2$ and verified \eqn{nu13x20}--\eqn{nu13x2j}.
We note that $x_{2,j}=0$ for $j > 91$. Now,
$$
x_{3,j} = \sum_{i\ge0} x_{2,i} b_{i,j},
$$
and we note that $x_{3,0}=0$.
We have
$$
\pi_{13}( x_{2,0} b_{0,j} ) = 1 + \pi_{13}(b_{0,j})
\ge 2 +  \lfloor \tfrac{1}{14}(13j - 8)\rfloor
$$
by Corollary \cor{13.7}. 
For $i\ge1$ 
\begin{align*}
\pi_{13}( x_{2,i} b_{i,j} ) &= \pi_{13}( x_{2,i}) 
+ \pi_{13}( b_{i,j}) \ge
3 +  \lfloor \tfrac{1}{14}(13i)\rfloor 
+  \lfloor \tfrac{1}{14}(13j-i+6)\rfloor\\
&\ge
3 +  \lfloor \tfrac{1}{14}(13j+12i-7)\rfloor
\ge
2 +  \lfloor \tfrac{1}{14}(13j-9)\rfloor,
\end{align*}
again by Corollary \cor{13.7}. It follows that
$$
\pi_{13}(x_{3,j}) \ge 2 +  \lfloor \tfrac{1}{14}(13j-9)\rfloor,
$$
and \eqn{nu13x2b1j} holds for $b=2$.
Now supposed $b\ge2$ is fixed and that \eqn{nu13x2b1j} holds.
We have
$$
x_{2b,j} = \sum_{i\ge1} x_{2b-1,i} a_{i,j}.
$$
Now
$$
\pi_{13}(x_{2b-1,1} a_{1,j}) 
= \pi_{13}(x_{2b-1,1}) + \pi_{13}(a_{1,j})
\ge 2b-2  + \pi_{13}(a_{1,j}) \ge 2b-1 + 
 \lfloor \tfrac{1}{14}(13j)\rfloor,
$$
by a direct calculation noting that $a_{1,j}=0$ for $j > 13$.
For $i\ge2$ 
\begin{align*}
\pi_{13}( x_{2b-1,i} a_{i,j} ) &= \pi_{13}( x_{2b-1,i}) 
+ \pi_{13}( a_{i,j}) \ge
2b-2 +  \lfloor \tfrac{1}{14}(13i-10)\rfloor 
+  \lfloor \tfrac{1}{14}(13j-i+13)\rfloor \\
&\ge
2b-2 +  \lfloor \tfrac{1}{14}(13j+12i-10)\rfloor
\ge
2b-1 +  \lfloor \tfrac{1}{14}(13j)\rfloor,
\end{align*}
again by Corollary \cor{13.7}. It follows that
$$
\pi_{13}(x_{2b,j}) \ge 2b-1 +  \lfloor \tfrac{1}{14}(13j)\rfloor,
$$
and \eqn{nu13x2bj} holds. For $i \ge 1$

Again suppose $b\ge2$ is fixed, and that \eqn{nu13x2bj} holds.
We have
$$
x_{2b+1,j} = \sum_{i\ge1} x_{2b,i} b_{i,j}.
$$
For $i \ge 1$
\begin{align*}
\pi_{13}( x_{2b,i} b_{i,j} ) &= \pi_{13}( x_{2b,i}) 
+ \pi_{13}( b_{i,j}) \ge
2b-1 +  \lfloor \tfrac{1}{14}(13i)\rfloor 
+  \lfloor \tfrac{1}{14}(13j-i+6)\rfloor\\
&\ge
2b-1 +  \lfloor \tfrac{1}{14}(13j+12i-8)\rfloor
\ge
2b +  \lfloor \tfrac{1}{14}(13j-10)\rfloor,
\end{align*}
again by Corollary \cor{13.7}. It follows that
$$
\pi_{13}(x_{2b+1,j}) \ge 2b +  \lfloor \tfrac{1}{14}(13j-10)\rfloor,
$$
and \eqn{nu13x2b1j} holds with $b$ replaced by $b+1$. 
Lemma \lem{13.8} follows by induction.
\end{proof}

\begin{corollary}
\label{cor:13.9}
For $c\ge 2$,
\beq                
 \mathbf{a}(13^{c}n + \gamma_{c}) - 13 \cdot  \mathbf{a}(13^{c-2}n + \gamma_{c-2}) \equiv 0
\pmod{13^{c-1}}.
\mylabel{eq:a13ccong}
\eeq           
For $a\ge1$,
\begin{align}
\spt(13^{a+2}n + \gamma_{a+2}) - 13 \cdot \spt(13^{a}n + \gamma_{a}) &\equiv 0
\pmod{13^{a+1}},                                   
\mylabel{eq:spt13abigcong}\\
\spt(13^{a}n + \gamma_{a}) &\equiv 0
\pmod{13^{\lfloor \frac{a+1}{2}\rfloor}}.
\mylabel{eq:spt13alittlecong}
\end{align}
\end{corollary}
We note that \eqn{a13ccong} holds when $c=2$ by taking $\gamma_0=1$.
Also when $a=0$ the congruence \eqn{spt13abigcong} has a stronger form.
The proof of the congruence
\beq
\spt(169n-7) - 13\cdot \spt(n) \equiv 0 \pmod{169}.
\mylabel{eq:spt169cong}
\eeq
is analogous to the proof of \eqn{spt25cong}.


\section{The spt-function modulo $\ell$} \mylabel{sec:sptell}

In this section we improve on results in \cite{Ga10a} and \cite{Br-Ga-Ma}
for the spt-function and the second moment rank function modulo $\ell$.
We let
$$
J_\ell(z) = \sum_{n=s_\ell}^\infty j_\ell(n) q^n,
$$
where $J_\ell(z)$ is defined in \eqn{Jelldef},
and
define
\beq
K_\ell(z) := G_\ell(z) + (-1)^{\frac{1}{2}(\ell-1)} \, \ell \,
             \sum_{n=\lceil\tfrac{s_\ell}{\ell}\rceil}^\infty 
              j_\ell(\ell n) q^n,
\mylabel{eq:Kelldef}
\eeq
where $G_\ell(z)$ is defined in \eqn{Gelldef}.
Then we have
\begin{theorem}
\mylabel{thm:Kellthm}
If $\ell \ge5$ is prime, then
$K_\ell(z)$ is an entire modular form of weight $(\ell+1)$ on the
full modular group $\SL_2(\Z)$.
\end{theorem}
\begin{proof}
Suppose $\ell\ge5$ is prime.
We utilize Serre's \cite[pp.223--224]{Se} 
results on the trace of a modular form on $\Gamma_0(\ell)$.
By Theorem \thm{Gell} we know that $G_\ell(z)$ is an entire
modular form of weight $(\ell+1)$ on $\Gamma_0(\ell)$.
By \cite[Lemma 7]{Se},
\beq
\Tr(G_\ell) = G_\ell + \ell^{1 - \frac{1}{2}(\ell+1)}\,
G_\ell \stroke W \stroke U
\mylabel{eq:Trace}
\eeq
is an entire modular form of weight $(\ell+1)$ on 
$\SL_2(\Z)$. See \cite[pp.223--224]{Se} for definition of $W$, $U$
and the notation used. From \eqn{Gelltrans} we find that
\beq
G_\ell \stroke W = (-1)^{\frac{1}{2}(\ell-1)} \, \ell^{\frac{1}{2}(\ell+1)}
\, J_\ell.
\mylabel{eq:GellW}
\eeq
From \eqn{Kelldef}, \eqn{Trace} and \eqn{GellW} we see that
$$
K_\ell = \Tr(G_\ell)
$$
is 
an entire modular form of weight $(\ell+1)$ on 
$\SL_2(\Z)$.
\end{proof}

We observed special cases of the following Corollary in 
\cite[Theorem 6.1]{Ga10a}.
\begin{corollary}
\mylabel{cor:sptmodell}
Suppose $\ell\ge5$ is prime. Then
\beq
\sum_{n=\lceil\tfrac{\ell}{24}\rceil}^\infty
\spt(\ell n - s_\ell) q^{n - \frac{\ell}{24}}
\equiv
\eta(z)^{r_\ell} \, L_\ell(z) \pmod{\ell}
\mylabel{eq:sptmodellres}
\eeq
for some integral entire modular form $L_\ell(z)$ on the full modular
group of weight
$\ell+1 - 12\lceil\tfrac{\ell}{24}\rceil$, and where $r_\ell$ and $s_\ell$
are defined in \eqn{rsdefs}.
\end{corollary}
\begin{proof}
Suppose $\ell\ge5$ is prime. Since
$$
(24n-1)\,p(n) \equiv 0 \pmod{\ell},
$$
for $24n\equiv1\pmod{\ell}$, and using Theorem \thm{Kellthm}
we have
$$
\frac{\eta(z)^{2\ell}}{\eta(\ell z)} 
\sum_{n=0}^\infty \mathbf{a}(\ell n - s_\ell) q^{n - \frac{\ell}{24}}
\equiv P_\ell(z) \pmod{\ell},
$$
for some integral $P_\ell(z)\in M_{\ell+1}(\Gamma(1))$.
We note that
$$
\spt(\ell n - s_\ell) \ne 0
$$
implies that $\ell n - s_\ell\ge 1$ and $n \ge \lceil\tfrac{\ell}{24}\rceil$.
It follows that 
$$
\frac{\eta(z)^{2\ell}}{\eta(\ell z)} 
\sum_{n=\lceil\tfrac{\ell}{24}\rceil}^\infty 
\spt(\ell n - s_\ell) q^{n - \frac{\ell}{24}}
\equiv \Delta(z)^c \, L_\ell(z) \pmod{\ell},
$$
where 
$\Delta(z)$ is Ramanujan's
function
\beq
\Delta(z) := \eta(z)^{24} = q \prod_{n=1}^\infty (1 - q^n)^{24},
\mylabel{eq:Deltadef}
\eeq
$c =  \lceil\tfrac{\ell}{24}\rceil$ and 
$L_\ell(z)$ is some integral modular form in $M_{\ell + 1 - 12c}(\Gamma(1))$.
Thus
$$
\sum_{n=\lceil\tfrac{\ell}{24}\rceil}^\infty 
\spt(\ell n - s_\ell) q^{n - \frac{\ell}{24}}
\equiv \Delta(z)^{c-\ell} \, L_\ell(z) \pmod{\ell},
$$
and the result follows since
$$
r_\ell = c - \ell.
$$
\end{proof}

We conclude the paper by improving a result in \cite{Br-Ga-Ma} for
the second rank moment function. From \eqn{sptid}
\beq
N_2(n) = 2n\,p(n) - 2\,\spt(n).
\mylabel{eq:N2id}
\eeq
We note that 
the analog of Corollary \cor{sptmodell} holds for the partition
function $p(n)$ except the weight is $2$ less. See either 
\cite[Theorem 3.4]{Ga10a} or \cite[Theorem3]{Ah-Bo}.                
This together with Corollary \cor{sptmodell} and \eqn{N2id} implies
\begin{corollary}
\mylabel{cor:N2modell}
Suppose $\ell\ge5$ is prime. Then
\beq
\sum_{n=\lceil\tfrac{\ell}{24}\rceil}^\infty
N_2(\ell n - s_\ell) q^{n - \frac{\ell}{24}}
\equiv
\eta(z)^{r_\ell} \left( Q_\ell(z) + L_\ell(z)\right) \pmod{\ell}
\mylabel{eq:N2modellres}
\eeq
for some integral entire modular forms $Q_\ell(z)$ and $L_\ell(z)$ 
on  the full modular
group of weights $k$ and $k+2$ respectively where
$k=\ell-1 - 12\lceil\tfrac{\ell}{24}\rceil$.
\end{corollary}

We illustrate Theorem \thm{Kellthm} and Corollaries \cor{sptmodell}
and \cor{N2modell} in the case $\ell=17$.
We find that
$$
K_{17}(z) = G_{17}(z) + 17\,\sum_{n=1}^\infty j_{17}(17n) q^n
 = -17\, E_6(z)^3 - 26148\,\Delta(z)\,E_6(z),
$$
$$
\sum_{n=0}^\infty \spt(17n + 5) q^{n + \frac{7}{24}} 
\equiv 
14 \, \eta(z)^7 \, E_6(z) \pmod{17},
$$
and
$$
\sum_{n=0}^\infty N_2(17n + 5) q^{n+\frac{7}{24}} 
\equiv 
\eta(z)^7 \left(2\,E_4(z) + 6\,E_6(z)\right) \pmod{17}.
$$
Here
$E_4(z)$ and $E_6(z)$ are the usual Eisenstein series
\beq
E_4(z) := 1 + 240 \sum_{n=1}^\infty \sigma_3(n) q^n, \qquad
E_6(z) := 1 - 504 \sum_{n=1}^\infty \sigma_5(n) q^n,
\mylabel{eq:E46def}
\eeq
where $\sigma_k(n) = \sum_{d\mid n} q^k$.



\noindent
\textbf{Acknowledgements}
I would like to thank Ken Ono and Zachary Kent for their comments
and suggestions.

\bibliographystyle{amsplain}

\end{document}